\documentclass[12pt]{amsart}

\usepackage[latin2]{inputenc}
\usepackage{t1enc}
\usepackage[mathscr]{eucal}
\usepackage{pgfplots}
\pgfplotsset{width=7cm,compat=1.8}
\usepackage{pgfplotstable}
\usepackage{caption}
\usepackage{graphicx}
\usepackage{graphics}
\numberwithin{equation}{section}
\usepackage[margin=3cm]{geometry}
\usepackage{hyperref} 
\usepackage{tikz}
\usepackage{pgfplots}
\pgfplotsset{ticks=none, compat=newest}
\usepackage{chngcntr}
\counterwithout{figure}{section}
\pgfplotsset{soldot/.style={color=black,only marks,mark=*}}
\pgfplotsset{holdot/.style={color=black,fill=white,only marks,mark=*}}
\usepackage{mathtools,array}
\usepackage{tikz}
\usetikzlibrary{calc,intersections,}
\hypersetup{
 colorlinks=true,       
    linkcolor=cyan,          
    citecolor=green,        
    filecolor=magenta,      
   urlcolor=cyan           
   }

\allowdisplaybreaks
\newcommand{\eq}{\begin{equation}}
\newcommand{\en}{\end{equation}}

\DeclareSymbolFont{bbold}{U}{bbold}{m}{n}
\DeclareSymbolFontAlphabet{\mathbbold}{bbold}
\newcommand{\ind}{\mathbbold{1}}

\makeatletter

\newcommand{\Rmnum}[1]{\expandafter\@slowromancap\romannumeral #1@}
\makeatother

\newtheorem{thm}{Theorem}[section]
\newtheorem{prop}[thm]{Proposition}
\newtheorem{lem}[thm]{Lemma}

\theoremstyle{definition}
\newtheorem{remark}[thm]{Remark}
\newtheorem{defn}[thm]{Definition}

\newtheorem{conjecture}[thm]{Conjecture}
\newtheorem{notation}[thm]{Notation}
\numberwithin{equation}{section}

\renewcommand{\and}{ \quad \text{and} \quad }

\begin{document}
\pgfmathsetseed{184}
\newcommand{\Emmett}[5]{
\draw[#4] (0,0)
\foreach \x in {1,...,#1}
{   -- ++(#2,rand*#3-0.5*#2*#2-0.75*#2)
}
node[right] {#5};
}

\title[]{Scalar conservation laws with white noise initial data}

\author{Mehdi Ouaki}
\address{Department of Statistics \#3860 \\ 451 Evans Hall \\ University of California at Berkeley \\ Berkeley, CA 94720-3860 \\ USA}
\email{mouaki@berkeley.edu}



\date{\today}

\begin{abstract} 
The statistical description of the scalar conservation law of the form $\rho_t=H(\rho)_x$ with $H: \mathbb{R} \rightarrow \mathbb{R}$ a smooth convex function has been an object of interest when the initial profile $\rho(\cdot,0)$ is random. The special case when $H(\rho)=\frac{\rho^2}{2}$ (Burgers equation) has in particular received extensive interest in the past and is now understood for various random initial conditions. We prove in this paper a conjecture on the profile of the solution at any time $t>0$ for a general class of Hamiltonians $H$ and show that it is a stationary piecewise-smooth Feller process. Along the way, we study the excursion process of the two-sided linear Brownian motion $W$ below any strictly convex function $\phi$ with superlinear growth and derive a generalized Chernoff distribution of the random variable $\text{argmax}_{z \in \mathbb{R}} (W(z)-\phi(z))$. Finally, when $\rho(\cdot,0)$ is a white noise derived from an \textit{abrupt} L\'evy process, we show that the structure of shocks of the solution is a.s discrete at any fixed time $t>0$ under some mild assumptions on $H$. 
\end{abstract}

\maketitle
\section{Introduction}
We are interested in the following conservation law problem 
\begin{equation}\label{scalar1}
\left\{
    \begin{array}{ll}
        \rho_t=(H(\rho))_x ~~ , \text{ for }~~ t>0, x \in \mathbb{R}\\
        \rho(x,0)=\xi(x) ~~ , ~~x \in \mathbb{R}
    \end{array}
\right.
\end{equation}
where $H$ is a $C^2$ strictly convex function with superlinear growth at infinity and $\xi$ is a white noise. A question of interest is to describe the law of the process $\rho(\cdot,t)$ at any given time $t>0$. \\

\subsection{Background}\leavevmode\\

There is a straightforward link between the scalar conservation law and the Hamilton-Jacobi PDE. Indeed, if one defines
\begin{align*}
u(x,t)=\int_{-\infty}^{x}\rho(y,t)dy
\end{align*}
and the potential 
\begin{align*}
U_0(x)=\int_{-\infty}^{x} \xi(y) dy
\end{align*}
then $u$ solves the PDE
\begin{equation}\label{HJ1}
u_t=H(u_x)
\end{equation}
and is determined by the Hopf-Lax formula (see \cite{evans10}[Theorem 4, Chapter 3.3])
\begin{equation}\label{hopf-lax}
u(x,t)=\sup_{y \in \mathbb{R}} \left(U_0(y)-tL\left(\frac{y-x}{t}\right)\right)
\end{equation}
where $L$ is the Legendre transform of $H$ defined as 
\begin{align*}
L(q)=\max\limits_{p \in \mathbb{R}} \left(qp-H(p)\right)
\end{align*}
The rightmost maximizer $y(x,t)$ in the equation \eqref{hopf-lax} is called the backward Lagrangian, and is directly linked to the entropy solution $\rho$ of the scalar conservation law \eqref{scalar1} by the Lax-Oleinik formula (see \cite{evans10}[Theorem 1, Chapter 3.4])
\begin{equation*}
\rho(x,t)=(H')^{-1}\left (\frac{y(x,t)-x}{t}\right)=L'\left(\frac{y(x,t)-x}{t}\right)
\end{equation*}
The reader may be familiar with this other form of the Hamilton-Jacobi PDE 
\begin{equation}\label{HJ2}
u_t+H(u_x)=0
\end{equation}
If we denote by $u$ a solution of \eqref{HJ2}, then it is easy to see that $\tilde{u}(x,t):=-u(x,t)$ verifies $\tilde{u}_t=\tilde{H}(\tilde{u}_x)$ for the Hamiltonian $\tilde{H}(\rho)=H(-\rho)$. We will thus only restrict ourselves to the version of the scalar conservation law in \eqref{scalar1}.\\

When the Hamiltonian $H$ takes the simple form $H(\rho)=\frac{\rho^2}{2}$, the scalar conservation law \eqref{scalar1} is called Burgers equation and is written $\rho_t=\rho \rho_x$. The Lax-Oleinik formula simplifies to
\begin{equation}\label{burgers-lo}
\rho(x,t)=\frac{y(x,t)-x}{t}
\end{equation}
The Burgers equation has seen an extensive interest when the initial data $\rho(\cdot,0)$ is random in the context of Burgers turbulence. We will present thereby the most relevant results in this area. 

\subsection{Burgers equation when $\rho(\cdot,0)$ is a Brownian white noise}\leavevmode \\

This is the case when the initial potential $U_0$ is expressed as
\begin{equation}\label{noise}
U_0(x)=\sigma B(x), ~~ x \in \mathbb{R}
\end{equation}
where $\sigma>0$ is a diffusion factor and $B$ is a two-sided standard linear Brownian motion. 
In a remarkable paper \cite{groeneboom89} with the aim of studying the global behavior of isotonic estimators, Groeneboom completely determined the statistics of the process
\begin{align*}
\left(V(a):=\sup \left\{ x \in \mathbb{R} : B(x)-(x-a)^2 \text{ is maximal}\right\}, a \in \mathbb{R}\right)
\end{align*}
He showed that this process is pure-jump with jump kernels expressed in terms of Airy functions. 
By the Hopf-Lax formula and \eqref{burgers-lo}, this process is related to the solution of the Burgers equation with Brownian white noise initial data. \\

More precisely, let $\rho_{\sigma}(x,t)$ be the entropy solution of Burgers equation when the initial potential is determined by \eqref{noise}. Since in the Burgers case the Hamiltonian enjoys the same scaling as the Brownian motion. It follows that for every $t>0$, the process $(\rho_{\sigma}(x,t) , x \in \mathbb{R})$ has the same law as $(\sigma^{\frac{2}{3}}t^{-\frac{1}{3}}\rho_1(x((\sigma t)^{-\frac{2}{3}},1) , x \in \mathbb{R})$. The following theorem gives a precise description of the law of the entropy solution at time $t=1$.
\begin{thm}[Groeneboom 89, \cite{groeneboom89}]\label{groeneboom}
The process $(\rho_{\frac{1}{\sqrt{2}}}(x,1),x\in \mathbb{R})$ is a stationary piecewise-linear Markov process with generator $\mathcal{A}$ acting on a test function $\varphi \in C_{c}^{\infty}(\mathbb{R})$ as 
\begin{equation*}
\mathcal{A}\varphi(y)=-\varphi'(y)+\int_{y}^{\infty} (\varphi(z)-\varphi(y))n(y,z)dz
\end{equation*}
The jump density $n$ is given by the formula 
\begin{align*}
n(y,z)=\frac{J(z)}{J(y)}K(z-y)~~ ,~~ z>y
\end{align*}
where $J$ and $Z$ are positive functions defined on the line and positive half-line respectively, whose Laplace transforms
\begin{align*}
j(q)=\int_{-\infty}^{\infty} e^{qy} J(y)dy, ~~~~~k(q)= \int_{0}^{\infty}e^{-qy} K(y)dy
\end{align*}
are meromorphic functions on $\mathbb{C}$ given by
\begin{align*}
j(q)=\frac{1}{\mathrm{Ai}(q)}, ~~~~ k(q)=-2\frac{d^2}{dq^2} \mathrm{log}\mathrm{Ai}(q)
\end{align*}
where $\mathrm{Ai}$ denotes the first Airy function.
\end{thm}
\begin{remark}
For general $t>0$, the process $(\rho_{\frac{1}{\sqrt{2}}}(x,t),x \in \mathbb{R})$ is also a stationary piecewise-linear Markov process with generator
\begin{equation*}
\mathcal{A}^t\varphi(y)=-\frac{1}{t} \varphi'(y)+\int_{y}^{\infty} t^{-\frac{1}{3}}n(yt^{\frac{1}{3}},zt^{\frac{1}{3}})(\varphi(z)-\varphi(y))dz
\end{equation*}
In particular, the linear pieces have slope $\displaystyle -\frac{1}{t}$. 
\end{remark}

\subsection{Burgers equation when $\rho(\cdot,0)$ is a spectrally negative L\'evy process} \leavevmode \\

A L\'evy process $(X_t)_{t \in \mathbb{R}}$ is a process with stationary independent increments and such that $X_0=0$. By spectrally negative L\'evy process, we mean a process that has only downward jumps. For the Burgers equation, Bertoin in \cite{bertoin98} proved a remarkable closure theorem for this class of initial data. We quote here his result.
\begin{thm}[Bertoin 98, \cite{bertoin98}]\label{bertoin}
Consider Burgers equation of the form $\rho_t+\rho\rho_x=0$ with initial data $\xi(x)$ which is a spectrally negative L\'evy process for $x \ge 0$ and $\xi(x)=0$ for $x<0$. Assume that the expected value of $\xi(1)$ is positive. Then for each fixed $t>0$, the backward Lagrangian $y(x,t)$ has the property that $(y(x,t)-y(0,t))_{x \ge 0}$ is independent of $y(0,t)$ and is in the parameter $x$ a subordinator, i.e. a nondecreasing L\'evy process. Its distribution is that of the first passage process 
\begin{align*}
x \mapsto \inf \{ z \ge 0 : t\xi(z)+z>x \}
\end{align*}
Furthermore, if we denote by $\psi(s)$ and $\Theta(t,s)$ $(s \ge 0)$ respectively the Laplace exponents of $\xi(x)$ and $y(x,t)-y(x,0)$,
\begin{align*}
\mathbb{E}[\mathrm{exp}(s\xi(x))]=\mathrm{exp}(x\psi(s)) \\
\mathbb{E}[\mathrm{exp}(s(y(x,t)-y(0,t)))]=\mathrm{exp}(x\Theta(t,s))
\end{align*}
 then we have the functional identity 
\begin{align*}
\psi(t\Theta(t,s))+\Theta(t,s)=s
\end{align*}
Moreover, the process $(\rho(x,t)-\rho(0,t))_{x \ge 0}$ is a L\'evy process, and its Laplace exponent $\psi(t,q)$ verifies the Burgers equation 
\begin{equation}\label{laplace-exp}
\psi_t+\psi\psi_q=0
\end{equation}
\end{thm}
\begin{remark}
This theorem is remarkable in the sense that it provides an infinite-dimensional, nonlinear dynamical system which preserves the independence and homogeneity properties of its random initial configuration. Moreover, it was observed in \cite{menonpego} that the evolution according to Burgers equation of the Laplace exponents in \eqref{laplace-exp} corresponds to a Smoluchowski coagulation equation \cite{smoluchowski} with additive rate which determines the jump statistics. This connection is simply due to the L\'evy-Khintchine representation of Laplace exponents.
\end{remark}

\subsection{Scalar conservation law with general Hamiltonian $H$ }\leavevmode \\

A natural question that arises is if the previous phenomenon (the entropy solution at later times having a simple form that can be explicitly described) is intrinsic to the Burgers equation or if the same holds for scalar conservation laws with general Hamiltonians $H$. In an attempt to answer this question, Menon and Srinivasan in \cite{menon10} proved that when the initial condition $\xi$ is a spectrally positive strong Markov process, then the entropy solution of \eqref{scalar1} at later times remains Markov and spectrally positive. However, it is not as clear whether the Feller property is preserved through time. The following conjecture was stated in that paper, together with different heuristic but convincing ways to see why that must be true.

\begin{conjecture}\label{conjec}
If the initial data $\xi$ of the scalar conservation law in \eqref{scalar1} is either
\begin{enumerate}
\item{A \textit{white noise} derived from a spectrally positive L\'evy process.}
\item{A stationary spectrally positive Feller process with bounded variation.}
\end{enumerate} 
then the solution $\rho(\cdot,t)$ for any fixed time $t>0$ is a stationary spectrally positive Feller process with bounded variation. Moreover, its jump kernel and drift verify an integro-differential equation.  
\end{conjecture} 

\begin{remark}
By a result of Courr\`ege (see \cite{applebaum}[Theorem 3.5.3]), the generator $\mathcal{A}$ of any spectrally positive Feller process with bounded variation takes the form
\begin{align*}
\mathcal{A} \varphi(y)=b(y)\varphi'(y)+\int_{y}^{\infty} (\varphi(z)-\varphi(y))n(y,dz)
\end{align*}
given that $C^{\infty}_{c}(\mathbb{R}) \subset \mathcal{D}(\mathcal{A})$ ($C^{\infty}_{c}(\mathbb{R})$ is the space of infinitely differentiable functions with compact support and $\mathcal{D}(\mathcal{A})$ is the domain of the generator $\mathcal{A}$). Moreover the kernel $n$ verifies the integrability condition : $\int_{y}^{\infty} (1 \wedge \vert y-z \vert^2)n(y,dz)<\infty$.
\end{remark}

A variant\footnote{Under some mild conditions on the Hamiltonian $H$, and a slight modification of the nature of the initial data.} of the second part of this conjecture when the initial data is a piecewise-deterministic spectrally positive Feller process was recently proved by Kaspar and Rezakhanlou in \cite{rezakhkaspar1} and \cite{rezakhkaspar}. We give below an explicit exposition of their result together with the exact form of the integro-differential equation verified by the drift and the jump kernel. This equation \eqref{kinetic} was formally derived by Menon and Srinivasan in \cite{menon10} and shown to be equivalent to the following Lax equation 
\begin{equation*}
\partial_t \mathcal{A}=[\mathcal{A},\mathcal{B}]=\mathcal{A}\mathcal{B}-\mathcal{B}\mathcal{A}
\end{equation*}
where $\mathcal{A}^{t}$ is the generator of $x \mapsto \rho(x,t)$ and $\mathcal{B}^{t}$ is the generator of $t \mapsto \rho(x,t)$. We give explicit formulas for these generators below in the statement of Theorem \ref{rezakh}.  \\

\begin{notation}
We write $\mathcal{M}_1$ for the set of probability measures on the real line, and
\begin{align*}
[H]_{y,z}=\frac{H(y)-H(z)}{y-z} \text{ for } y \ne z
\end{align*}
\end{notation}
\begin{thm}[Kaspar and Rezakhanlou 20, \cite{rezakhkaspar}]\label{rezakh}
Assume that the initial data $\rho^0=\rho^0(x)$ is  zero of $x<0$ and is a Markov process for $x \ge 0$ that starts at $\rho^{0}(0)=0$. More precisely, its infinitesimal generator $\mathcal{A}^0$ has the form
\begin{align*}
\mathcal{A}^{0}\varphi(\rho_{-})=b^{0}(\rho_{-})\varphi'(\rho_{-}) +\int_{\rho_{-}}^{\infty} (\varphi(\rho_{+})-\varphi(\rho_{-}))f^{0}(\rho_{-},\rho_{+})d\rho_{+}
\end{align*}
Furthermore, assume that 
\begin{enumerate}
\item{The rate kernel $f^{0}(p{-},p_{+})$ is $C^1$ and is supported on 
\begin{align*}
\{ (p_{-},p_{+}) : P_{-} \le p_{-} \le p_{+} \le P_{+} \}
\end{align*}
for some constants $P_{\pm}$.}
\item{The Hamiltonian function $H : [P_{-},P_{+}] \to \mathbb{R}$ is $C^2$, convex, has positive right-derivative at $p=P_{-}$ and finite left-derivative at $p=P_{+}$.}
\item{The initial drift $b^0$ is $C^1$ and satisfies $b^0 \le 0$ with $b^0(\rho)=0$ whenever $\rho \notin [P_{-},P_{+}]$.}
\end{enumerate}
Then for each fixed $t>0$, the process $x\mapsto \rho(x,t)$ (where $\rho$ is a solution of \eqref{scalar1}) has $x=0$ marginal given by $\ell^{0}(d\rho_0,t)$ where $\ell^{0}: [0,\infty) \to \mathcal{M}_1$ is the unique function such that $\ell^{0}(d\rho,0)=\delta_0(d\rho)$ and 
\begin{align*}
\frac{d\ell^{0}(d\rho,t)}{dt}=(\mathcal{B}^{t*} \ell^{0}(\cdot,t))(d\rho,t)
\end{align*}
where $\mathcal{B}^{t*}$ is the adjoint operator of $\mathcal{B}^{t}$, that acts on measures with 
\begin{align*}
\mathcal{B}^{t}\varphi(\rho_{-})=-H'(\rho_{-})b(\rho_{-},t)\varphi'(\rho_{-})-\int_{\rho_{-}}^{\infty} [H]_{\rho_{-},\rho_{+}}(\varphi(\rho_{+})-\varphi(\rho_{-}))f(\rho_{-},\rho_{+},t)d\rho_{+}
\end{align*} 
for any test function $\varphi$. Moreover the process $x\mapsto \rho(x,t)$ evolves for $0<x<\infty$ according to a Markov process with generator $\mathcal{A}^{t}$ given by 
\begin{align*}
\mathcal{A}^{t}\varphi(\rho_{-})=b(\rho_{-},t)\varphi'(\rho_{-}) +\int_{\rho_{-}}^{\infty} (\varphi(\rho_{+})-\varphi(\rho_{-}))f(\rho_{-},\rho_{+},t)d\rho_{+}
\end{align*}
Here $b$ and $f$ are obtained from their initial conditions 
\begin{align*}
b(\rho,0)=b^{0}(\rho), ~~~~ f(\rho_{-},\rho_{+},0)=f^{0}(\rho_{-},\rho_{+})
\end{align*}
$b$ solves the ODE with parameter 
\begin{equation*}
\partial_t b(\rho,t)=H''(\rho)b(\rho,t)^2
\end{equation*}
and $f$ solves the following Boltzmann-like kinetic equation 
\begin{align}\label{kinetic}
\partial_t f(\rho_{-},\rho_{-},t)=
Q(f,f)+C(f)+\partial_{\rho_{-}} (fV_{\rho_{-}}(\rho_{-},\rho_{+},t))+\partial_{\rho_{+}}(fV_{\rho_{+}}(\rho_{-},\rho_{+},t))
\end{align}
where the velocities $V_{\rho_{-}}$ and $V_{\rho_{+}}$ are given by 
\begin{align*}
V_{\rho_{-}}(\rho_{-},\rho_{+},t)=([H]_{\rho_{-},\rho_{+}}-H'(\rho_{-}))b(\rho_{-},t)\\
V_{\rho_{+}}(\rho_{-},\rho_{+},t)=([H]_{\rho_{-},\rho_{+}}-H'(\rho_{+}))b(\rho_{+},t)
\end{align*}
the coagulation-like collision kernel $Q$ is 
\begin{align*}
Q(f,f)(\rho_{-},\rho_{+},t)=\int_{\rho_{-}}^{\rho_{+}}([H]_{\rho_{*},\rho_{+}}-[H]_{\rho_{-},\rho_{*}})f(\rho_{-},\rho_{*},t)f(\rho_{*},\rho_{+},t)d\rho_{*} \\
-\int_{\rho_{+}}^{\infty}([H]_{\rho_{-},\rho_{+}}-[H]_{\rho_{+},\rho_{*}})f(\rho_{-},\rho_{+},t)f(\rho_{+},\rho_{*},t)d\rho_{*} \\
-\int_{\rho_{-}}^{\infty} ([H]_{\rho_{-},\rho_{*}}-[H]_{\rho_{-},\rho_{+}})f(\rho_{-},\rho_{+},t)f(\rho_{-},\rho_{*},t) d\rho_{*}
\end{align*}
and the linear operator $C$ is given by
\begin{align*}
C(f)(\rho_{-},\rho{+})=f(\rho_{-},\rho_{+})(b(\rho_{-},t)H''(\rho_{-})-([H]_{\rho_{-},\rho_{+}}-H'(\rho_{-}))\partial_{\rho_{-}}b(\rho_{-},t))
\end{align*}
\end{thm}

The purpose of this paper is to prove the first part of the conjecture when the initial data $\xi$ is a Brownian \textit{white noise} and thus extend the results of Groeneboom \cite{groeneboom89} in the Burgers case. We show that at any fixed time $t>0$, the solution $\rho(\cdot,t)$ is a stationary piecewise-smooth Feller process and we give an explicit description of its generator. This result proves the complete integrability of scalar conservation laws for this class of initial data and moves away from the unnatural emphasis on Burgers equation. Our method as will be seen by the reader can be extended when the white noise is derived from a spectrally positive L\'evy process with non-zero Brownian exponent. Our shortcoming in this case will be not having explicit formulas for the jump kernel. We also show that the structure of shocks of Burgers turbulence holds for the general scalar conservation law under the assumption of rough initial data.\\

Since the entropy solution is expressed via the Lax-Oleinik formula. It is natural to study the law of the process $\Psi^{\phi}$ defined as 

\begin{equation}\label{psi}
\Psi^{\phi}(x)=\sup \left\{ y \in \mathbb{R} : U_0(y)-\phi(y-x) =\max\limits_{z \in \mathbb{R}} \left( U_0(z)-\phi(z-x) \right) \right\} ,~~ x \in \mathbb{R}
\end{equation}

where $U_0$ is a spectrally positive L\'evy process and $\phi$ is a $C^2$ strictly convex function with superlinear growth, such that $U_0(y)=o(\phi(y))$\footnote{We write $f=o(g)$ if $\lim \frac{f}{g}=0$ and $f=O(g)$ if $\frac{f}{g}$ is bounded.} for $\vert y \vert \to \infty$. The relationship between the process $\Psi^{\phi}$ and the entropy solution $\rho(\cdot,t)$ of \eqref{scalar1} is the following 
\begin{align*}
\rho(x,t)=L'\left(\frac{\Psi^{tL(\frac{\cdot}{t})}(x)-x}{t}\right)
\end{align*}\\

Our paper is organized as follows

\begin{enumerate}
\item{In Section \ref{prelim}, we give some preliminary results on the process $\Psi^{\phi}$ when $U_0$ is a spectrally positive L\'evy process such as its Markovian property.}
\item{In Section \ref{sect3}, we will focus on the case where $U_0$ is a two-sided Brownian motion and show that the process $\Psi^{\phi}$ is pure jump, following similar ideas used by Groeneboom in \cite{groeneboom89}. The main ingredient being the path decomposition of Markov processes when they reach their ultimate maximum. This result implies that the Brownian motion $U_0$ has excursions below the sequence of convex functions $(x \mapsto \phi(x-x_n))_{ n \in \mathbb{N}}$ where $(x_n)_{n \in \mathbb{N}}$ are the jump times of the process $\Psi^{\phi}$ (which is a discrete set by a result of Section \ref{sect5}). However, the justification of many manipulations used in \cite{groeneboom89} rely on the regularity and asymptotic properties of Airy functions at infinity, as those arise naturally in the expressions of transition densities used throughout the study of the Brownian motion with parabolic drift. Unfortunately, those special functions are intrinsic to this special case as we will explain later, and one do not have similar expressions in the general case.}
\item{In Section \ref{sect4}, we circumvent this difficulty by using a more analytic approach to prove the smoothness and integrability of the densities that were used in Section \ref{sect3}. Moreover, via Girsanov theorem we manage to express explicitly the jump kernel of the process $\Psi^{\phi}$ in terms of the distribution of Brownian excursion areas. Along the way, we find the joint density of the maximum and its location of the process $(W(z)-\phi(z))_{z \in \mathbb{R}}$ where $W$ is a two-sided Brownian motion. In particular, the density of $\text{argmax}_{z\in \mathbb{R}} (W(z)-\phi(z))$ enjoys a simple expression similar to Chernoff distribution for the parabolic drift.}
\item{ Finally, in Section \ref{sect5} we give a sufficient condition on the L\'evy process $U_0$ for the process $\Psi^{\phi}$ to have discrete range (with the convention that a set is discrete if it is countable with no accumulation points). As a consequence, this implies that the structure of shocks of the entropy solution $\rho(\cdot,t)$ is discrete for any time $t>0$ when the initial data belongs to the large class of \textit{abrupt} L\'evy processes introduced by Vigon in \cite{vigon}, this result generalizes the findings of Bertoin \cite{bertoin-shocks} and Abramson \cite{abramson} when $U_0$ is spectrally positive.}\\
\end{enumerate}

We give here our main results \\

\begin{thm}\label{scalar}
Suppose that the initial potential $U_0$ is a two-sided Brownian motion and let $\rho$ be the solution of the scalar conservation law $\rho_t=(H(\rho))_x$. Then for every fixed $t>0$, the process $x \mapsto \rho(x,t)$ is a stationary piecewise-smooth Feller process. Its generator is given by 
\begin{align*}
\mathcal{A}^{t}\varphi(\rho_{-})=-\frac{\varphi'(\rho_{-})}{tH''(\rho_{-})}+\int_{\rho_{-}}^{\infty} (\varphi(\rho_{+})-\varphi(\rho_{-})) n(\rho_{-},\rho_{+},t)d\rho_{+}
\end{align*}
for any test function $\varphi \in C_{c}^{\infty}(\mathbb{R})$, where 
\begin{equation}\label{kernel-formula}
n(\rho_{-},\rho_{+},t)=\frac{(\rho_{+}-\rho_{-})K(\rho_{-},\rho_{+},t)}{\sqrt{2\pi t (H'(\rho_{+})-H'(\rho_{-}))^3}}
 \frac{\rho_{+}+\int_{\rho_{+}}^{\infty} \frac{H''(\rho)-K(\rho_{+},\rho,t)}{\sqrt{2\pi t (H'(\rho)-H'(\rho_{+}))^3}}d\rho}{\rho_{-}+\int_{\rho_{-}}^{\infty} \frac{H''(\rho)-K(\rho_{-},\rho,t)}{\sqrt{2\pi t (H'(\rho)-H'(\rho_{-}))^3}}d\rho}
\end{equation}
for $\rho_{-}<\rho_{+}$, and 
\begin{align*}
K(\rho_{-},\rho_{+},t)=H''(\rho_{+})
\text{exp} \left (-\frac{t}{2}  \int_{\rho_{-}}^{\rho_{+}} \rho_{*}^2 H''(\rho_{*})d\rho_{*} \right) \times \\
\mathbb{E}\left[\text{exp} \left(-\int_{\rho_{-}}^{\rho_{+}} \textbf{e}(tH'(\rho_{*}))d\rho_{*} \right)\right]
\end{align*}
where $\textbf{e}$ is a Brownian excursion on the interval $[tH'(\rho_{-}),tH'(\rho_{+})]$. 
\end{thm}
\begin{remark}\leavevmode
\begin{enumerate}
\item{
The profile of the solution at any fixed time $t>0$ is a concatenation of smooth pieces that evolve as solutions of ODEs with vector field (or drift) $b(\rho,t):=-\frac{1}{tH''(\rho)}$ and are interrupted by stochastic upward jumps distributed via the jump kernel $n(\cdot,\cdot,t)$. We prove in Section \ref{sect5} that in the Brownian white noise case, under mild assumptions on the Hamiltonian $H$, the set of jump times is discrete, i.e. : there are only a finite number of jumps on any given compact interval.}
 \item{For any $\epsilon>0$, the profile of $x \mapsto \rho(x,\epsilon)$ is a piecewise-deterministic Markov process and belongs to the class of initial data considered in the second part of the Conjecture \ref{conjec}. A consequence of this observation would be that the kernel $(\rho_{-},\rho_{+},t) \mapsto n(\rho_{-},\rho_{+},t)$ in the expression \eqref{kernel-formula} verifies the kinetic equation \eqref{kinetic}. However, Theorem \ref{rezakh} only considers a variant of the original statement of the conjecture as it forces the initial data to be flat on the negative real-line (whereas here we deal with a stationary process) and restricts the range of $\rho^0$ on a compact interval $[P_{-},P_{+}]$. These technical modifications arise from the very challenging proof of existence and uniqueness of a classical solution to \eqref{kinetic} under general assumptions. Verifying that the kernel $n$ in the Brownian white noise case is a solution to the kinetic equation \eqref{kinetic} from the explicit expression \eqref{kernel-formula} seems also inaccessible at the present due to the complicated term involving the Brownian excursion. This verification was done for the Burgers case by Menon and Srinivasan in \cite{menon10}[Section 6] through many non-trivial calculations, but relied extensively on the connection with Airy functions and an associated Painlev\'e property. }
 \end{enumerate}
\end{remark}

The following result is a consequence of our study of the process $\Psi^{\phi}$. It gives an explicit formula for the density of the random variable $\text{argmax}_{\omega \in \mathbb{R}} (W(\omega)-\phi(\omega))$ where $W$ is a two-sided Brownian motion. From results of Section \ref{sect4}, we also have access to the joint distribution of 
\begin{align*}
 (\text{argmax}_{\omega \in \mathbb{R}} (W(\omega)-\phi(\omega)),\max_{\omega \in \mathbb{R}} (W(\omega)-\phi(\omega)))
 \end{align*}
 but we omit it here because the expression is quite large. 

\begin{thm}\label{chernoff}
Let $\omega_M$ be the location of the maximum of the process $(S(\omega)=W(\omega)-\phi(\omega))_{\omega \in \mathbb{R}}$ where $W$ is a two-sided Brownian motion, its density is equal to
\begin{align*}
\frac{\mathbb{P}[\omega_M \in dt]}{dt}=\frac{1}{2}f^{\phi}(t)f^{\phi(-\cdot)}(-t) 
\end{align*}
for any $t \in \mathbb{R}$, and where 
\begin{align*}
f^{\phi}(t)=\phi'(t)+\int_{0}^{\infty} \frac{1-p^{\phi}(t,u)}{\sqrt{2\pi u^3}} du
\end{align*}
with 
\begin{align*}
p^{\phi}(t,u)=\text{exp} \left (-\frac{1}{2} \int_{t}^{t+u} \phi'(z)^2 dz \right ) \mathbb{E}\left[\text{exp}\left (-\int_{t}^{t+u} \phi''(z) \textbf{e}(z)dz \right)\right]\text{ for } u>0
\end{align*}
where $\textbf{e}$ is a Brownian excursion on $[t,t+u]$.
\end{thm}
\begin{remark}
In the parabolic drift case (Chernoff distribution), the term $\phi''$ is constant and the Laplace transform of a standard Brownian excursion area is known to be expressed via Airy functions. We will develop on the connection between the formulas found by Groeneboom in \cite{groeneboom89} and ours at the end of Section \ref{sect4}. Also, we refer the reader to the survey \cite{janson} for a more detailed exposition on the distribution and Laplace transform of various Brownian paths areas.
\end{remark}

\begin{figure}
\begin{center}
\begin{tikzpicture}[scale=1.7]
\begin{axis}[
          xmin=-1,xmax=10,
          ymin=-1,ymax=10,
          axis lines=middle,
          restrict y to domain=-2:10,
          enlargelimits,
          xlabel={\tiny{$x$}},
          ylabel={\tiny{$\rho(\cdot,t)$}},
          ticklabel style={fill=white}]
\addplot[domain=-3:3,mark=none] {2+0.1*(3-x)^2} node[fill=white, below]{};
\addplot[domain=3:8,mark=none] {4+0.1*(8-x)^2} node[fill=white, below]{};
\addplot[domain=8:12,mark=none] {6+0.1*(12-x)^2} node[fill=white, below]{};
\addplot[holdot] coordinates{(3,2)(8,4)};
\addplot[soldot] coordinates{(3,6.5)(8,7.6)};
\draw [dashed,help lines] (3,6.5) -| (3,0);
\draw [dashed,help lines] (8,7.6) -| (8,0);
\node[] at (3,-0.7) {\tiny{$x_1$}};
\node[] at (8,-0.7) {\tiny{$x_2$}};
\node[] at (0,1.7) {\tiny{Drift  }$b(\cdot,t)$};
\end{axis}
\end{tikzpicture}
\caption{The typical profile of the entropy solution at a given time $t>0$.}
\end{center}
\end{figure}
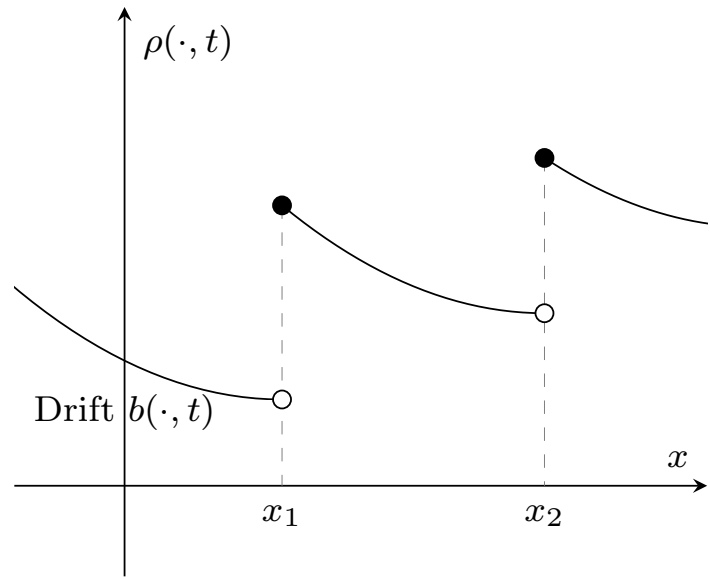

We define now a class of rough L\'evy processes called \textit{abrupt} that were introduced by Vigon in \cite{vigon}.
\begin{defn}
A L\'evy process $(X_t)_{t \in \mathbb{R}}$ is said to be \textit{abrupt} if its paths have unbounded variation and almost surely for all local maxima $m$ of $X$ we have 
\begin{align*}
\liminf\limits_{h \downarrow 0} \frac{1}{h}(X_{m-h}-X_{m-})=+\infty \text { and } \limsup\limits_{h \downarrow 0} \frac{1}{h}(X_{m+h}-X_m)=-\infty 
\end{align*}
\end{defn}
\begin{remark}
A L\'evy process $X$ with paths of unbounded variation is \textit{abrupt} if and only if 
\begin{align*}
\int_{0}^{1} t^{-1} \mathbb{P} \left[X_t \in [at,bt] \right]dt < \infty ,~~ \forall a<b
\end{align*}
Examples of \textit{abrupt} L\'evy processes include stable processes with index $\alpha \in (1,2]$ and any process with non-zero Brownian exponent.
\end{remark}
Our last main result determines the structure of shocks of the scalar conservation law when the initial data is a white noise derived from an \textit{abrupt} L\'evy process. 

\begin{thm}\label{discrete}
Assume that the L\'evy process $U_0$ is spectrally positive, \textit{abrupt} and is such that $U_0(y)=O(\vert y \vert)$ for $\vert y \vert \rightarrow \infty$, then the set 
\begin{align*}
\mathcal{L}^{t}=\{ y \in \mathbb{R} : y=y(x,t) \text{ or } y=y(x-,t) \text{ for some } x \in \mathbb{R} \}
\end{align*}
is almost surely discrete for any fixed time $t>0$. We say then that the structure of shocks of the entropy solution $\rho(\cdot,t)$ is discrete.  
\end{thm}
\begin{remark}
From a point of view of hydrodynamic turbulence, a discontinuity of the entropy solution $\rho(\cdot,t)$ at position $x$ means the presence of a cluster of particles at this location at time $t$. Those clusters interact with each other via inelastic shocks, and the cluster at location $x$ and at time $t$ contains all the particles that were initially located in $[y(x-,t),y(x,t))$. Our result shows that at any given time $t>0$, the set of clusters is discrete. When the initial data is a L\'evy white noise, we can picture that there are infinitely many particles initially scattered everywhere with i.i.d velocities. Therefore, when we assume that this initial profile is rough (as it is the case when the potential $U_0$ is an \textit{abrupt} L\'evy process), this turbulence forces all the particles to aggregate in heavy disjoint lumps instantaneously for any time $t>0$.
\end{remark}

\subsection*{Acknowledgement}
I would like to thank my advisor Fraydoun Rezakhanlou for many fruitful discussions.
\section{Preliminaries}\label{prelim}

\begin{notation}
We will use the notation $\text{argmax}^{+} f$ to denote the rightmost maximizer of a function $f$ (i.e. : the last time at which a function $f$ reaches its maximum).
\end{notation}

Menon and Srinivisan proved in their paper \cite{menon10} a closure theorem for white noise initial data for the scalar conservation law solutions. They showed that if initially the potential $U_0$ is spectrally positive with independent increments then $\rho(\cdot,t)$ is a spectrally positive Markov process for any fixed $t>0$. The proof of this statement follows from standard use of path decomposition of strong Markov processes at their ultimate maximum. The same holds for our process $\Psi^{\phi}$. Precisely, we have the following theorem for which we give the proof for the sake of completeness.

\begin{thm}\label{psi}
Assume that $U_0$ is a spectrally positive L\'evy process, then the process $\Psi^{\phi}$ is a non-decreasing Markov process. Moreover for any $a \in \mathbb{R}$, the process $\Psi^{\phi}(.+a)-a$ has the same distribution as $\Psi^{\phi}$. 
\end{thm}
\begin{proof}
For $x_1 \le x_2$ and $y \le \Psi^{\phi}(x_1)$, we have that 
\begin{align*}
U_0(\Psi^{\phi}(x_1))-U_0(y) & \ge \phi(\Psi^{\phi}(x_1)-x_1)-\phi(y-x_1) \\
& \ge \phi(\Psi^{\phi}(x_1)-x_2)-\phi(y-x_2)
\end{align*}
By the convexity of $\phi$, and hence $\Psi^{\phi}(x_1) \le \Psi^{\phi}(x_2)$. Also, by definition $\Psi^{\phi}$ is a c\`adl\`ag process (right continuous with left hand limits). Take $h>0$, then 
\begin{equation}\label{independence}
\Psi^{\phi}(x+h)=\Psi^{\phi}(x)+\\
\text{argmax}^{+}_{y \ge 0} \left(U_0(y+\Psi^{\phi}(x))-\phi(y+\Psi^{\phi}(x)-(x+h) )\right )
\end{equation}
The process $U^{x}(y):=U_0(y)-\phi(y-x)$ is clearly Markov. By Millar's theorem of path decomposition of Markov processes when they reach their ultimate maximum (see \cite{millar}), the process $(U^{x}(y+\Psi^{\phi}(x))_{y \ge 0}$ is independent of $(U^{x}(y))_{y \le \Psi^{\phi}(x)}$ given $(\Psi^{\phi}(x),U^{x}(\Psi^{\phi}(x)))$ (because of the upward jumps of $U_0$, the maximum is attained at the right hand limit). Moreover, because of the independence of the increments of $U^x$, the process $(U^{x}(y+\Psi^{\phi}(x))-U^{x}(\Psi^{\phi}(x)))_{y \ge 0}$ is independent of $(U^{x}(y))_{y \le \Psi^{\phi}(x)}$ given $\Psi^{\phi}(x)$. Now it suffices to see that $(\Psi^{\phi}(y))_{y \le x}$ only depends on the pre-maximum process $(U^{x}(y))_{y \le \Psi^{\phi}(x)}$ because of the monotonicity of $\Psi^{\phi}$, this fact alongside the equation \eqref{independence} gives the Markov property of the process $\Psi^{\phi}$. The last statement follows easily from the stationarity of increments of $U_0$. 
\end{proof}
\begin{remark}
Notice that except in the last statement, the stationarity of increments was not used in the proof of the Markovian property of the process $\Psi^{\phi}$.
\end{remark}
\section{The process $\Psi^{\phi}$ in the Brownian case}\label{sect3}
In this section, we assume that $W:=U_0$ is a two-sided Brownian motion. We proved in the previous section that the process $\Psi^{\phi}$ is Markov and enjoys a space-time shifted stationarity property. Hence, we shall only determine its transition function at time zero and consequently the form of its generator at this time. In this section we will differentiate and switch the order of integrals and differentiations without any justification, as Section \ref{sect4} is devoted to take care of all those technicalities.

\begin{notation}
In the sequel, we will deal with functions of the form $f(s,x,t,y)$ where $t$ and $s$ play the role of temporal variables, and $x$ and $y$ that of spatial variables. Without confusion, the notation $\partial_x f(s,x,t,y)$ (resp. $\partial_y f(s,x,t,y)$) refer to the partial derivative of $f$ with respect to the second variable (resp. fourth variable).
\end{notation}

We state here the first result regarding the transition function of the process $\Psi^{\phi}$. 
\begin{thm}
Let $h>0$ and $\omega_1<\omega_2$ be two real numbers. Then we have that
\begin{align*}
\mathbb{P}[\Psi^{\phi}(h)\in d\omega_2 \vert \Psi^{\phi}(0)=\omega_1]=\mathbb{P}[\text{argmax}^{+}_{\omega \ge \omega_1} X^{h}(\omega) \in d\omega_2]
\end{align*}
where $X^{h}(\omega):=S^{\downarrow}(\omega)+r^{h}(\omega)$ and 
\begin{itemize}
\item{ $(S^{\downarrow}(\omega))_{\omega \ge \omega_1}$ is the Markov process $(S(\omega):=W(\omega)-\phi(\omega))_{\omega \ge \omega_1}$ started at zero and Doob-conditioned to stay negative (i.e to hit $-\infty$ before $0$). Precisely, its transition function is given by
\begin{equation}\label{transition}
\mathbb{P}[S^{\downarrow}(t) \in dy \vert S^{\downarrow}(s)=x ]=\frac{\mathbb{P}[\tau_{0}=\infty \vert S(t)=y]}{\mathbb{P}[\tau_{0}=\infty \vert S(s)=x]}f(s,x,t,y)dy 
\end{equation}
for $t>s>\omega_1$ and $x,y<0$, and where $\tau_0$ is the first hitting time of zero of the process $S$. The function $f$ is the transition density of the process $S$ killed at zero, at time $t$ and state $y$, formally defined as 
\begin{align*}
\mathbb{P}[S(t) \in dy, \max\limits_{s \le u \le t} S(u)<0 \vert S(s)=x]=f(s,x,t,y)dy
\end{align*}
Moreover, the entrance law of $S^{\downarrow}$ is given by
\begin{equation}\label{entrance}
\mathbb{P}[S^{\downarrow}(t) \in dy ]=\frac{\mathbb{P}[\tau_{0}=\infty \vert S(t)=y]}{\partial_x \mathbb{P}[\tau_{0}=\infty \vert S(s)=x]_{|x=0}}\partial_x f(\omega_1,0,t,y)dy
\end{equation}}
\item{The function $r^{h}$ is defined as $r^{h}(\omega)=\phi(\omega)-\phi(\omega-h)+c$ where $c$ is a constant such that $r^{h}(\omega_1)=0$.}
\end{itemize}
\end{thm}
\begin{proof}
We have that 
\begin{align*}
\mathbb{P}[\Psi^{\phi}(h)\in d\omega_2 \vert \Psi^{\phi}(0)=\omega_1]&=\mathbb{P}[\text{argmax}^{+}_{\omega \ge \omega_1} (W(\omega)-\phi(\omega-h))\in d\omega_2\vert \Psi^{\phi}(0)=\omega_1]\\
&=\mathbb{P}[\text{argmax}^{+}_{\omega \ge \omega_1} (S(\omega)-S(\omega_1)+r^{h}(\omega)) 
\in d\omega_2 \\
 \vert \text{argmax}^{+} S(\omega)=\omega_1]
\end{align*}
Now, using Millar path decomposition of Markov processes when they reach their ultimate maximum, the expression of the transition densities of the post-maximum process in \cite{millar}[Equation 9] on the process $S$, and the spatial homogeneity of the Brownian motion (and thus of $S$), we get \eqref{transition}. To get the entrance law it suffices to send $s$ to $\omega_1$ and $x$ to zero. 
\end{proof}

Let us now introduce some notation to keep our formulas compact. 
\begin{notation}
Denote by 
\begin{align*}
J(s,x)=\mathbb{P}[\tau_0=\infty \vert S(s)=x]=\mathbb{P}[ S(u)<0 \text{ for all } u \ge s \vert S(s)=x],~~x<0
\end{align*}
and define 
\begin{align*}
j(s,x)= \frac{\partial }{\partial x} J(s,x) ,~~s\in\mathbb{R},~~ x<0
\end{align*}
\begin{align*}
j(s)=\lim_{x \uparrow 0} \frac{\partial }{\partial x} J(s,x),~~ s\in \mathbb{R}
\end{align*}
Also denote 
\begin{align*}
\Phi(s,x,\omega)=\frac{\mathbb{P}[\tau_0 \in d\omega \vert S(s)=x]}{d\omega},~~ s < \omega , ~~x \in \mathbb{R}
\end{align*}
Furthermore, let $\tilde{S}$ be the process defined as $(\tilde{S}(\omega):=W(\omega)-\phi(-\omega))_{\omega \in \mathbb{R}}$. We define $\tilde{f}$ and $\tilde{\Phi}$ analogously. 
\end{notation}
With this notation, the entrance law of the process $S^{\downarrow}$ is expressed as 
\begin{equation}\label{entrance-1}
\mathbb{P}[S^{\downarrow}(t)\in dy]=\frac{J(t,y)}{j(s)}\partial_x f(\omega_1,0,t,y)dy ,~~ t > \omega_1 \text{ and } y<0
\end{equation}
The next result will allow us to recover the transition function of the process $\Psi^{\phi}$.
\begin{thm}\label{main}
Let $\omega_1<\omega_2$ and $x^{*} \in (0,r^{h}(\omega_2))$. Define $\omega^{*} \in (\omega_1,\omega_2)$  to be the unique point such that $r_h(\omega^*)=x^{*}$ (such a time exists because of the strict convexity of $\phi$ that makes $r_h$ strictly increasing). Then we have that 
\begin{align*}
\frac{\mathbb{P}[\text{argmax}^{+}_{\omega \ge \omega_1} X^{h}(\omega) \in d\omega_2, \max\limits_{\omega \ge \omega_1} X^{h}(\omega) \in dx^{*}]}{d\omega_2 dx^{*}}=\\
2\int_{-\infty}^{x_{*}} \frac{j(\omega_2-h)}{j(\omega_1)}\Phi(\omega^*-h,y-x^{*},\omega_2-h)\tilde{\Phi}(-\omega^{*},y-x^{*},-\omega_1)dy
\end{align*}
\end{thm}

Before proving this theorem, we will state a lemma that links the joint distribution of the maximum of the diffusion $S$ and its location with the functionals $f$ and $J$ . 
\begin{lem}\label{joint}
Let $M$ and $\omega_{M}$ be respectively the maximum of the process $(S(\omega))_{\omega \ge s}$ and its location, we have then that 
\begin{equation}
\frac{\mathbb{P}[\omega_M \in dt, M \in dz \vert S(s)=x]}{dt dz}=\frac{1}{2}j(t)\partial_y f(s,x-z,t,0)=-j(t)\Phi(s,x-z,t)
\end{equation}
\end{lem}
\begin{proof}
We have by the Markov property that
\begin{align*}
\mathbb{P}[\omega_M >t, M \in dz \vert S(s)=x]&=\mathbb{P}[\max\limits_{s \le u \le t} S(u) < z, \max\limits_{u \ge t} S(u) \in dz \vert S(s)=x]\\
&=\int_{-\infty}^{z} f(s,x-z,t,y-z)\mathbb{P}[\max_{u \ge t } S(u) \in dz \vert S(t)=y]dy
\end{align*}
Now we see that
\begin{equation*}
\mathbb{P}[\max_{u \ge t } S(u) \in dz \vert S(t)=y]=J(t,y-z-dz)-J(t,y-z)=-j(t,y-z)dz
\end{equation*}
Hence 
\begin{equation*}
\mathbb{P}[\omega_M >t, M \in dz \vert S(s)=x]=-\int_{-\infty}^{0}f(s,x-z,t,y)j(t,y)dydz
\end{equation*}
Thus
\begin{equation}\label{time-der}
\frac{\mathbb{P}[\omega_M \in dt, M\in dz]}{dzdt}=\frac{\partial}{\partial t} \int_{-\infty}^{0} f(s,x-z,t,y)j(t,y)dy 
\end{equation}
Now, by Kolmogorov forward and backward equations on the diffusion $S$ we have that
\begin{equation*}
\partial_t f = \phi'(t)\partial_y f + \frac{1}{2} \partial^2_y f 
\end{equation*}
and
\begin{equation*}
\partial_t j=\phi'(t)\partial_y j-\frac{1}{2} \partial^2_y j
\end{equation*}
By interchanging the time partial derivative and the integral sign in \eqref{time-der}, we find by integration by parts
\begin{align*}
\frac{\partial}{\partial t} \int_{-\infty}^{0} f(s,x-z,t,y)j(t,y)dy =\phi'(t)\left [ fj\right]_{-\infty}^{0}+\frac{1}{2}\left[j\partial_y f-f\partial_y j \right]_{-\infty}^{0}
\end{align*}
Now it suffices to see that $f$ vanishes at both zero and infinity, from which the first equality follows. For the second equality, it suffices to see that
\begin{align*}
\mathbb{P}[\tau_0 > t \vert S(s)=x]=\int_{-\infty}^{0} f(s,x,t,y)dy
\end{align*} 
Differentiating with respect to time and using the Kolmogorov forward equation in the same fashion as was done before gives the result.
\end{proof}
\begin{remark}
All these differentiations and integrations by parts are justified by the fact that $f$ and $j$ are sufficiently smooth and integrable away from $\{t=s\}$. This fact will be proved in the next section.
\end{remark}
\begin{proof}[Proof of Theorem \ref{main}]
We have that
\begin{align*}
\mathbb{P}[\text{argmax}^{+}_{\omega \ge \omega_1} X^{h}(\omega) \in d\omega_2, \max\limits_{\omega \ge \omega_1} X^{h}(\omega) \in dx^{*}]=\\
\int_{-\infty}^{x^{*}} \mathbb{P}[X^{h}(\omega^{*}) \in dy,\text{argmax}^{+}_{\omega \ge \omega_1} X^{h}(\omega) \in d\omega_2, \max\limits_{\omega \ge \omega_1} X^{h}(\omega) \in dx^{*}]
\end{align*}
Because for $ \omega \in [\omega_1,\omega^{*})$, we have that  $X^{h}(\omega) \le r_h(\omega) < x^{*}$, then by the Markov property we get that
\begin{align*}
\mathbb{P}[X^{h}(\omega^{*}) \in dy,\text{argmax}^{+}_{\omega \ge \omega_1} X^{h}(\omega) \in d\omega_2, \max\limits_{\omega \ge \omega_1} X^{h}(\omega) \in dx^{*}]=\\
\mathbb{P}[X^{h}(\omega^{*})\in dy]\mathbb{P}[\text{argmax}^{+}_{\omega \ge \omega^{*}} X^{h}(\omega) \in d\omega_2, \max\limits_{\omega \ge \omega^{*}} X^{h}(\omega) \in dx^{*} \vert X^{h}(\omega^{*})=y]
\end{align*}
Let us focus first on the second term of this product. The law of the Markov process $X^h$ is that of the process $S+r^{h}$ conditioned to stay below $r_h$. However, when $X^{h}$ starts from the state $y<x^{*}$ at time $\omega^{*}$, the event we condition on has positive probability and hence it is just the naive conditioning. Thus, we can write
\begin{align*}
\mathbb{P}\left[\text{argmax}^{+}_{\omega \ge \omega^{*}} X^{h}(\omega) \in d\omega_2, \max\limits_{\omega \ge \omega^{*}} X^{h}(\omega) \in dx^{*} \vert X^{h}(\omega^{*})=y\right]=\\
\mathbb{P}\left[\text{argmax}^{+}_{\omega \ge \omega^{*}} S(\omega)+r^{h}(\omega) \in d\omega_2, \max\limits_{\omega \ge \omega^{*}} S(\omega)+r^h(\omega) \in dx^{*} \right.\\
\left. \vert S(\omega^{*})=y-x^{*} , S(\omega) \le 0 \text { for all } \omega \ge \omega^{*}\right]\\
\end{align*}
This probability is equal to the ratio of this probability
\begin{align*}
\mathbb{P}_1=\mathbb{P}\left[\text{argmax}^{+}_{\omega \ge \omega^{*}} S(\omega)+r^{h}(\omega) \in d\omega_2, \max\limits_{\omega \ge \omega^{*}} S(\omega)+r^h(\omega) \in dx^{*}, \right.\\
\left. S(\omega) \le 0 \text { for all } \omega \ge \omega^{*} \vert S(\omega^{*})=y-x^{*}
 \right]
\end{align*}
over the probability 
\begin{align*}
\mathbb{P}_2=\mathbb{P}\left[S(\omega) \le 0 \text { for all } \omega \ge \omega^{*} \vert S(\omega^{*})=y-x^{*}
 \right]
\end{align*}
For the first probability $\mathbb{P}_1$, notice that on the event that $\{ \max\limits_{\omega \ge \omega^{*}} S(\omega)+r^h(\omega) \in dx^{*} \}$, we always have that $S(\omega) \le 0 \text{ for all } \omega \ge \omega^{*}$, because $r^{h}(\omega) \ge x^{*}$ for $\omega \ge \omega^{*}$. Thus
\begin{align*}
\mathbb{P}_1=\mathbb{P}\left[\text{argmax}^{+}_{\omega \ge \omega^{*}} S(\omega)+r^{h}(\omega) \in d\omega_2, \max\limits_{\omega \ge \omega^{*}} S(\omega)+r^h(\omega) \in dx^{*} \right. 
\left. \vert S(\omega^{*})=y-x^{*}
\right]
\end{align*}
Now we have that
\begin{align*}
S(\omega)+r^{h}(\omega)=W(\omega)-\phi(\omega-h)+c ~~,~~ \omega \ge \omega^{*}
\end{align*}
Hence
\begin{align*}
(S(\omega)+r^h(\omega) \vert S(\omega^*)=y-x^{*})_{\omega \ge \omega^{*}}  \overset{\mathrm{d}}{=} (S(\omega-h)  \vert S(\omega^{*}-h)=y)_{\omega \ge \omega^{*}}
\end{align*}
Thus by using Lemma \ref{joint} for $s=\omega^{*}-h$ and $x=y-x^{*}$, we get that
\begin{align*}
\mathbb{P}_1=-j(\omega_2-h)\Phi(\omega^{*}-h,y-x^{*},\omega_2-h)d\omega_2 dx^{*}
\end{align*}
Therefore
\begin{equation}\label{firstpart}
\frac{\mathbb{P}_1}{\mathbb{P}_2}=-\frac{j(\omega_2-h)\Phi(\omega^*-h,y-x^*,\omega_2-h)}{J(\omega^*,y-x^*)}d\omega_2dx^{*}
\end{equation}
Finally for the first term $\mathbb{P}[X^{h}(\omega^*) \in dy]$, we have that 
\begin{align*}
\mathbb{P}[X^{h}(\omega^*) \in dy]&=\mathbb{P}[S^{\downarrow}(\omega^{*}) \in d(y-r^h(\omega^{*}))]\\
&=\mathbb{P}[S^{\downarrow}(\omega^{*}) \in d(y-x^{*})]\\
&=\frac{J(\omega^{*},y-x^{*})}{j(\omega_1)}\partial_x f(\omega_1,0,\omega^{*},y-x^{*})dy
\end{align*}
Now it is not hard to see that we have the following equality
\begin{equation}\label{backward}
\tilde{f}(s,x,t,y)=f(-t,y,-s,x)
\end{equation}
This is true because both those functions verify the same PDE with the same boundary and growth conditions, by combining the backward and forward Kolmogorov equations. Hence 
\begin{equation*}
\partial_x f(s,x,t,y)=\partial_y \tilde{f}(-t,y,-s,x)
\end{equation*} 
Hence, by Lemma \ref{joint} 
\begin{align*}
\partial_x f(\omega_1,0,\omega^*,y-x^*)&=\partial_y \tilde{f}(-\omega^*,y-x^{*},-\omega_1,0)\\
&=-2\tilde{\Phi}(-\omega^{*},y-x^{*},-\omega_1)
\end{align*}
Thus 
\begin{equation}\label{secondpart}
\mathbb{P}[X^{h}(\omega^*)\in dy\vert X^h(\omega_1)=0]=-2\frac{J(\omega^*,y-x^*)}{j(\omega_1)}\tilde{\Phi}(-\omega^*,y-x^*,-\omega_1)dy
\end{equation}
Multiplying equations \eqref{firstpart} and \eqref{secondpart} and integrating with respect to $y$ on $(-\infty,x^{*})$ gives the result. 
\end{proof}
\begin{figure}
\begin{center}
\begin{tikzpicture}[scale=1.7]
\draw[->] (-1,0)--(4,0) node[right]{\tiny{$\omega$}};
\draw[->] (0,-1)--(0,3) node[above]{\tiny{$y$}};
\draw[domain=0:3.5,variable=\x] plot({\x},{0.3*\x*\x});
\draw[dashed] (0,1)--(4,1);
\draw[dashed,color=blue] (2.468,0)--(2.468,1);
\draw[dashed,color=blue] (1.82,0)--(1.82,1);
\node[] at (2.3,2) {\tiny{$r^{h}$}};
\node[] at (1,-1) {\tiny{$X^{h}$}};
\node[] at (2.468,-0.2) {\tiny{$\omega_2$}};
\node[] at (1.70,0.1) {\tiny{$\omega^{*}$}};
\node[] at (-0.2,1) {\tiny{$x^{*}$}};
\node[] at (-0.1,0.1) {\tiny{$\omega_1$}};
\Emmett{370}{0.01}{0.16}{black}{}
\end{tikzpicture}
\caption{Path decomposition of $X^{h}$ at its maximum}
\end{center}
\end{figure}
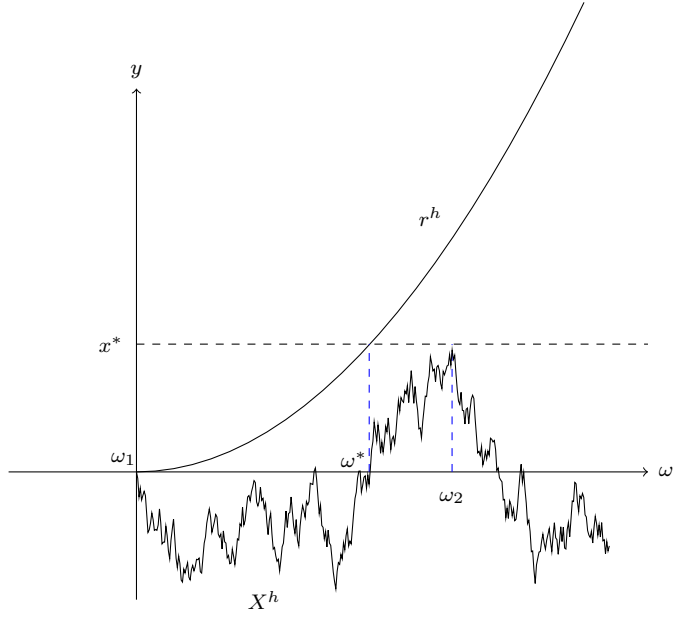

We are ready now to state the main result of this section.

\begin{thm}
The transition function of the process $\Psi^{\phi}$ is given by 
\begin{align*}
\mathbb{P}[\Psi^{\phi}(h) \in d\omega_2 \vert \Psi^{\phi}(0)=\omega_1]=
2\frac{j(\omega_2-h)}{j(\omega_1)} \times \\
\int_{\omega_1}^{\omega_2}\int_{\infty}^{0} (r^{h})'(\omega)\Phi(\omega-h,y,\omega_2-h)\tilde{\Phi}(-\omega,y,-\omega_1)dy d\omega
\end{align*}
Moreover, the process $\Psi^{\phi}$ is pure-jump and its generator at zero is given by its action on any test function $\varphi \in C_{c}^{\infty}(\mathbb{R})$ 
\begin{align*}
\mathcal{A}^{\phi} \varphi(y)=\int_{y}^{\infty} (\varphi(z)-\varphi(y))n^{\phi}(y,z)dz
\end{align*}
where 
\begin{align*}
n^{\phi}(y,z)=2\frac{j(z)}{j(y)}\int_{y}^{z}\int_{-\infty}^{0} \phi''(\omega)\Phi(\omega,x,z)\tilde{\Phi}(-\omega,x,-y)dxd\omega =:\frac{j(z)}{j(y)}K^{\phi}(y,z)
\end{align*}
\end{thm}
\begin{proof}
By integrating the formula in Theorem \ref{main} with respect to $x^*$ between and $0$ and $r^{h}(\omega_2)$ (as $X^h$ is pointwise at most $r^h$), we get that

\begin{align*}
\mathbb{P}[\text{argmax}^{+}_{\omega \ge \omega_1} X^h(\omega) \in d\omega_2]=
2\frac{j(\omega_2-h)}{j(\omega_1)}\int_{0}^{r^h(\omega_2)}\int_{-\infty}^{x^*}\Phi(\omega^{*}-h,y-x^*,\omega_2-h) \times \\
\tilde{\Phi}(-\omega^*,y-x^*,-\omega_1)dydx^*
\end{align*}
Now it suffices to do the change of variables $y'=y-x^*$ and $\omega=(r^h)^{-1}(x^*)$ to get the transition density. As for the generator part, it suffices to do the following Taylor expansion for $h \rightarrow 0$
\begin{align*}
(r^h)'(\omega)=\phi''(\omega) h +O(h^2)
\end{align*}
\end{proof}
\begin{remark}
In the next section, we will greatly simplify this expression of the generator by giving explicit formulas of $K^{\phi}$ and $j$ in Proposition \ref{kernel} and Theorem \ref{jou} respectively.  
\end{remark}
\section{Regularity of the transition functions and explicit formulas}\label{sect4}
The goal of this section is to prove the regularity of the transition density $f(s,x,t,y)$ away from the line $\{t=s\}$, so that we can justify all the operations we did in the previous section and to deduce along the way explicit formulas for the jump kernel of the process $\Psi^{\phi}$.\\

Processes such as the three-dimensional Bessel process, the three-dimensional Bessel bridges, and the Brownian motion killed at zero will be mentioned in some of the results of this section. We refer the unfamiliar reader to \cite{yorrevuz}[Chapters 3,6,11] for basic facts about these processes. \\

The following proposition gives a closed formula for the density $f$.
\begin{prop}\label{expression-f}
Let $x,y<0$ and $t >s$, the density $f$ is given by the formula 
\begin{align*}
f(s,x,t,y)=G(s,x,t,y)\text{exp}\left(-\phi'(t)y+\phi'(s)x-\frac{1}{2}\int_{s}^{t} (\phi'(u))^2du \right )\times\\
\mathbb{E}\left[\text{exp} \left(-\int_{s}^{t} \phi''(u)B(u)du\right) \vert B(s)=-x, B(t)=-y\right]
\end{align*}
where $B$ is a three-dimensional Bessel process, and $G$ is the transition density function of the Brownian motion killed at zero, given explicitly by 
\begin{equation*}
G(s,x,t,y)=\frac{1}{\sqrt{2\pi(t-s)}}\left(e^{-\frac{(x-y)^2}{2(t-s)}}-e^{-\frac{(x+y)^2}{2(t-s)}}\right)
\end{equation*}
\end{prop}
\begin{proof}
The process $S$ can be expressed as 
\begin{equation*}
S(t)=W(t)-\phi(t)=W(t)-\int_{s}^{t}\phi'(u)du -\phi(s)
\end{equation*}
Thus by Girsanov theorem, $S$ is a Brownian motion under the measure $\mathbb{Q}$ with Radon-Nikodym derivative given by
\begin{equation*}
\frac{d \mathbb{Q}}{d \mathbb{P}} _{| \mathcal{F}_t}=\text{exp} \left (\int_{s}^{t} \phi'(u)dW_u -\frac{1}{2} \int_{s}^{t} (\phi'(u))^2 du \right)
\end{equation*}
where $\mathcal{F}_t:=\sigma \{ W(u) : s \le u \le t \}$ is the canonical filtration of $W$. Thus for any function $F$ we have that 
\begin{equation*}
\mathbb{E}[F(S(t))\ind_{\{\max\limits_{s \le u \le t} S(u)<0\}} \vert S(s)=x]=\mathbb{E}[Z(t)F(W(t))\ind_{\{\max\limits_{s \le u \le t} W(u)<0\}} \vert W(s)=x]
\end{equation*}
where 
\begin{equation*}
Z(t):=\exp \left (-\int_{s}^{t} \phi'(u)dW_u-\frac{1}{2} \int_{s}^{t} \phi'(u)^2 du \right)
\end{equation*}
In particular for $F=F_{\epsilon}:=\frac{1}{2\epsilon} \ind_{[y-\epsilon,y+\epsilon]}$, we have that 
\begin{align*}
f(s,x,t,y)&=\lim_{\epsilon \rightarrow 0} \mathbb{E}[F_{\epsilon}(S(t))\ind_{\{\max\limits_{s \le u \le t} S(u)<0\}} \vert S(s)=x]\\
&=\lim_{\epsilon \rightarrow 0} \frac{1}{2\epsilon} \int_{y-\epsilon}^{y+\epsilon} \mathbb{E}[Z(t)\ind_{\{W(t) \in dz, \max\limits_{s \le u \le t} W(u)<0\}} \vert W(s)=x]
\end{align*}
Now if we denote by $W^{\partial}$ the Brownian motion killed at zero whose law is defined as 
\begin{align*}
\mathbb{E}[F(W^{\partial}(t))\vert W^{\partial}(s)=x]=\mathbb{E}[F(W(t))\ind_{\{\max_{s \le u \le t} W(u)<0\}} \vert W(s)=x]
\end{align*}
Thus 
\begin{align*}
f(s,x,t,y)&=\lim_{\epsilon \rightarrow 0} \frac{1}{2\epsilon} \int_{y-\epsilon}^{y+\epsilon} \mathbb{E}[Z^{\partial}(t) \vert W^{\partial}(t)=y,W^{\partial}(s)=x]p_{t-s}^{\partial}(x,z)dz \\
&=p_{t-s}^{\partial}(x,y)\mathbb{E}[Z^{\partial}(t) \vert W^{\partial}(t)=y, W^{\partial}(s)=x]
\end{align*}
where $Z^{\partial}$ is the same as $Z$ with $W$ replaced by $W^{\partial}$, and $p_{t}^{\partial}(x,y)$ is the transition density function of the process $W^{\partial}$. However it is a well-known fact that $p_{t-s}^{\partial}(x,y)=G(s,x,t,y)$, and the law of the Brownian motion killed at zero between $s$ and $t$ conditioned on its extreme values is the law of the reflection of a three-dimensional Bessel bridge between $(s,-x)$ and $(t,-y)$ (as our killed Brownian motion stays negative and the Bessel bridges are by definition positive). Finally, by using an integration by parts we have that
\begin{equation*}
d(B(u)\phi'(u))=\phi'(u)dB(u)+\phi''(u)B(u)du
\end{equation*}
Integrating between $s$ and $t$, we get the desired result. 
\end{proof}
\begin{remark}
From the last proposition, one can readily see that for fixed $s$ and $x$
\begin{align*}
0 \le f(s,x,t,y) \le C(t)e^{-A(t)y^2} \text{ for all } y 
\end{align*}
where $C$ and $A$ are locally bounded, and $A$ is locally bounded from below by a positive constant.
\end{remark}
Let us now prove that $f$ is smooth. First of all, one can extend $f$ to the positive line as well by defining
\begin{align*}
f(s,x,t,y)=-f(s,x,t,-y), ~~ y>0
\end{align*}
Then $f$ verifies in the distribution sense the following PDE (Kolmogorov forward equation) 
\begin{equation}\label{pde}
\partial_t f-\frac{1}{2} \partial^{2}_y f=\phi'(t)\partial_yf \text{ on } (t,y) \in (s,+\infty) \times \mathbb{R}
\end{equation}
and with boundary conditions $f(s,x,s,\cdot)=\delta_x - \delta_{-x}$, and obviously $f(s,x,t,0)=0$. Now, it is well-known that the function $G$ that we defined in Proposition \ref{expression-f} verifies the heat equation
\begin{align*}
\partial_t G -\frac{1}{2} \partial_y ^2 G =0
\end{align*}
with the same boundary conditions as $f$. Moreover, if one defines the function $\hat{G}$ as
\begin{align*}
\hat{G}(s,x,t,y)=\frac{1}{\sqrt{2\pi (t-s)}} e^{-\frac{(x-y)^2}{2(t-s)}}
\end{align*}
it is also a solution for the heat equation but with boundary condition $\hat{G}(s,x,s,\cdot)=\delta_x$. Thus, in order to study the regularity properties of the solution to \eqref{pde}, one might use Duhamel's principle to get a representation formula for $f$. More precisely, we shall prove the following theorem 
\begin{thm}\label{representation}
Fix $s,x \in \mathbb{R}$. There exists a function $h \in C([s,+\infty),L^1(\mathbb{R}) \cap L^{\infty}(\mathbb{R}))$ (where here $L^1(\mathbb{R}) \cap L^{\infty}(\mathbb{R})$ is the space of continuous functions on the real line that are uniformly bounded and absolutely integrable), such that 
\begin{align*}
h(t,y)=\int_{s}^{t} \int_{\mathbb{R}} \phi'(u)\hat{G}(u,z,t,y)\partial_zG(s,x,u,z)dzdu\\
-\int_{s}^{t}\int_{\mathbb{R}}\phi'(u)\partial_z \hat{G}(u,z,t,y)h(u,z)dzdu
\end{align*}
Furthermore, $h$ is smooth. 
\end{thm}
\begin{proof}
Let us fix $T>s$. Define the functional $\Xi^T$ from $\mathcal{C}_T:=C([s,T],L^1(\mathbb{R})\cap L^{\infty}(\mathbb{R}))$ into itself equipped with the norm
\[
||h||_{\mathcal{C}_T}:=\sup_{s \le t \le T} ||h(t)||_{L^{\infty}}+||h(t)||_{L^1}
\]
by
\begin{align*}
\Xi^T[h](t,y)=\int_{s}^{t} \int_{\mathbb{R}} \phi'(u)\hat{G}(u,z,t,y)\partial_zG(s,x,u,z)dzdu\\
- \int_{s}^{t}\int_{\mathbb{R}}\phi'(u)\partial_z\hat{G}(u,z,t,y)h(u,z)dzdu
\end{align*}
It is clear that $\Xi^T$ sends $\mathcal{C}_T$ to itself due to the growth rate of the Green functions $G$ and $\hat{G}$ at infinity in space. Moreover we have that for any two functions $h$ and $\tilde{h}$ in $\mathcal{C}_T$
\begin{align*}
|| \Xi^T[h](t,.)-\Xi^T[\tilde{h}](t,.) ||_{L^1} \le \int_{s}^{t} \vert \phi'(u)\vert du \int_{\mathbb{R}} |h(u,z)-\tilde{h}(u,z)|dz  \int_{\mathbb{R}} |\partial_z\hat{G}(u,z,t,y)|dy
\end{align*}
Now we see that 
\begin{align*}
\partial_z\hat{G}(u,z,t,y)=\frac{y-z}{\sqrt{2\pi (t-u)^3}} e^{-\frac{(y-z)^2}{2(t-u)}}\end{align*}
Hence 
\begin{align*}
\int_{\mathbb{R}} |\partial_z \hat{G}(u,z,t,y)|dy \le \frac{2}{\sqrt{2\pi(t-u)^3}} \int_{0}^{\infty} \omega e^{-\frac{\omega^2}{2(t-u)}}d\omega=\frac{2}{\sqrt{2\pi(t-u)}}
\end{align*}
Thus 
\begin{align*}
|| \Xi^T[h](t,.)-\Xi^T[\tilde{h}](t,.) ||_{L^1} \le \frac{4\sqrt{T-s} \sup_{u \in [s,T]} | \phi'(u) |}{\sqrt{\pi}} ||h-\tilde{h}||_{\mathcal{C}_T}
\end{align*}
A similar bound is found for the $L^{\infty}$ norm. Thus, for $T$ close enough to $s$, the operator $\Xi^{T}$ becomes a contraction, and thus by Picard theorem, it admits a unique fixed point. \\
Now define
\begin{align*}
T^*=\sup \{ T \ge s : \exists h \in \mathcal{C}_T \text{ such that } \Xi^T[h]=h \}
\end{align*}
Suppose that $T^{*}<\infty$, then it is easy to see by Gronwall inequality that for any sequence $(t_m)_{m \in \mathbb{N}}$ such that $t_m \uparrow T^*$, the sequence $(h(t_m,\cdot))_{m \in \mathbb{N}}$ is Cauchy in $L^{\infty}(\mathbb{R}) \cap L^1(\mathbb{R})$ and thus converge strongly to a unique limit that we denote $h(T^{*},\cdot)$. This extension thus belongs to $\mathcal{C}_{T^*}$. However, for small $\epsilon>0$, one can further extend the fixed point $h$ to $\mathcal{C}_{T^*+\epsilon}$ by the same contraction argument. This contradicts the definition of $T^*$, and thus $T^*=\infty$ from which follow the existence of a global solution. The smoothness of $h$ follows readily from that of the Green function $\hat{G}$ and the dominated convergence theorem. 
\end{proof}
We are now ready to prove the following result 
\begin{thm}
The function $f-G$ is everywhere smooth in the variables $(t,y)$, in particular the function $f$ is smooth away from $\{t=s\}$. 
\end{thm}
\begin{proof}
Define the function $q$ by 
\begin{align*}
q(s,x,t,y)=h(t,y)+G(s,x,t,y)
\end{align*}
where $h$ is the global solution from Theorem \ref{representation}.
By integration by parts we have that
\begin{align*}
q(s,x,t,y)=G(s,x,t,y)+\int_{s}^{t} \int_{\mathbb{R}} \phi'(u)\hat{G}(u,z,t,y)\partial_zG(s,x,u,z)dzdu+\\
\int_{s}^{t} \int_{\mathbb{R}} \phi'(u)\hat{G}(u,z,t,y) \partial_zh(u,z)dudz\\
= G(s,x,t,y)+\int_{s}^{\infty} \int_{\mathbb{R}} \phi'(u) \ind_{\{t \in (u,+\infty)\}}\hat{G}(u,z,t,y)\partial_z q(u,z) du dz
\end{align*}
Now it suffices to see that
\begin{align*}
(\partial_t -\frac{1}{2} \partial^2_{y})(\ind_{t \in (u,+\infty)} \hat{G}(u,z,t,y))=\delta_0(t-u)\hat{G}(u,z,u,y)=\delta_0(t-u)\delta_0(y-z)
\end{align*}
and thus the function $q$ verifies the PDE \eqref{pde} with the boundary conditions $q(s,x,s,.)=\delta_x-\delta_{-x}$. The result now would follow if we can prove that $f=q$. Consider the function $v:=f-q$, it verifies the PDE \eqref{pde} with vanishing initial condition. The growth condition of $v$ at infinity in space ensures that $v$ can be viewed as a tempered distribution. By taking the Fourier transform in space in the PDE \eqref{pde} we get that 
\begin{align*}
\partial_t \hat{v}(t,k)=\left(-\frac{1}{2}k^2+i\phi'(t)k\right)\hat{v}(t,k)
\end{align*}
Thus 
\begin{align*}
\partial_t(\hat{v}(t,k)e^{\frac{1}{2}k^2t-i\phi(t)k})=0
\end{align*}
which means that the distribution $\hat{v}(t,k)e^{\frac{1}{2}k^2 t-i\phi(t) k}$ is constant along the time variable $t$. Moreover, we also have that
\begin{align*}
\lim_{t \rightarrow s} v(t,\cdot)=0
\end{align*}
in the tempered distribution sense. Indeed for any $\varphi$ in the Schwartz space $\mathcal{S}(\mathbb{R})$, if we denote by $S^{\partial}$ is the diffusion $S$ killed at zero we have that
\begin{align*}
\lim_{t \to s} \int_{\mathbb{R}} \varphi(y)v(t,y)dy=\lim_{t \to s} \left [\left(\mathbb{E}[\varphi(S^{\partial}(t)) \vert S^{\partial}(s)=x]-\mathbb{E}[\varphi(W^{\partial}(t)) \vert W^{\partial}(s)=x]\right ) \right.\\
\left.-\left(\mathbb{E}[\varphi(-S^{\partial}(t)) \vert S^{\partial}(s)=x]-\mathbb{E}[\varphi(W^{\partial}(t)) \vert W^{\partial}(s)=-x]\right )\right.\\
\left.-\int_{\mathbb{R}} \varphi(y)h(t,y)dy\right]=0
\end{align*}
as $h(s,\cdot)=0$ and by using the dominated convergence theorem. Thus by continuity of the Fourier transform, one deduces that $v$ is zero everywhere, and hence $q=f$ as desired.
\end{proof}
Let us introduce now a function that is going to play a fundamental role in our calculations. Define $g$ by 
\begin{equation}\label{expression-g}
g(s,x,t,y)=G(s,x,t,y)\mathbb{E}\left[\text{exp}\left(-\int_{s}^{t} \phi''(u)B(u)du \right) \vert B(s)=-x,B(t)=-y\right]
\end{equation}
for $x,y<0$ and $t \ge s$, where $B$ is a three-dimensional Bessel process. Because $f$ is smooth away from $\{t=s\}$, the same holds for $g$. We have then the following lemma.
\begin{lem}
The function $g$ verifies the following PDE 
\begin{equation}\label{pde-g}
\partial_t g=\frac{1}{2} \partial^2_{y} g +\phi''(t)yg
\end{equation}
for $(t,y) \in (s,+\infty)\times(-\infty,0)$.
\end{lem}
\begin{proof}
We can replace the Bessel process $B$ by the Brownian motion killed at zero $W^{\partial}$ in the expression of $g$ in \eqref{expression-g} for the same reasons we gave earlier. Now let $\varphi \in C_{c}^{\infty}((s,+\infty)\times(-\infty,0))$ be a test function. We apply Ito formula to the following semi-martingale 
\begin{align*}
Y(t)=\varphi(t,W(t))\text{exp}\left(\int_{s}^{t} \phi''(u)W(u)du\right)
\end{align*}
where $W$ is a Brownian motion started at $x$.
We get then
\begin{align*}
dY(t)=\partial_y \varphi(t,W(t))\text{exp}\left(\int_{s}^{t} \phi''(u)W(u)du\right)dW(t)+\\
\left(\partial_t \varphi(t,W(t))+\frac{1}{2} \partial^2_y \varphi(t,W(t))+\varphi(t,W(t))\phi''(t)W(t)\right)
\text{exp}\left(\int_{s}^{t} \phi''(u)W(u)du\right)dt
\end{align*}
We integrate between $s$ and $t \wedge \tau_0$ (where $\tau_0$ is the first hitting time of zero of $W$). As the first term is a bounded local martingale (and hence a true martingale), by taking the expectation we get that 
\begin{align*}
\mathbb{E}\left[\varphi(t\wedge \tau_0,W(t \wedge \tau_0))]=\mathbb{E}[\int_{s}^{t\wedge \tau_0} \left(\partial_t \varphi(u,W(u))+\frac{1}{2} \partial^2_y \varphi(u,W(u)) \right. \right.\\
\left. \left. +\varphi(u,W(u))\phi''(u)W(u)\right)\text{exp}\left(\int_{s}^{u} \phi''(\omega)W(\omega)d\omega\right)du \right]
\end{align*}
Therefore 
\begin{align*}
\mathbb{E}\left[\varphi(t\wedge \tau_0,W(t \wedge \tau_0))]=\mathbb{E}[ \int_{s}^{t}\ind_{\{\max\limits_{s \le z \le u} W(z) <0\}} \left(\partial_t \varphi(u,W(u))+\frac{1}{2} \partial^2_y \varphi(u,W(u)) \right. \right.\\
\left. \left. +\varphi(u,W(u))\phi''(u)W(u)\right)\text{exp}\left(\int_{s}^{u} \phi''(\omega)W(\omega)d\omega\right)du\right]\\
=\int_{s}^{t} \mathbb{E} \left[ \left(\partial_t \varphi(u,W^{\partial}(u))+\frac{1}{2} \partial^2_y \varphi(u,W^{\partial}(u)) \right. \right. \\
\left. \left. +\varphi(u,W^{\partial}(u))\phi''(u)W^{\partial}(u)\right)\text{exp}\left(\int_{s}^{u} \phi''(\omega)W^{\partial}(\omega)d\omega\right)\right]du
\end{align*}
By sending $t \rightarrow \infty$ and conditioning on the value of $W^{\partial}(u)$, we get 
\begin{align*}
\int_{s}^{\infty} \int_{-\infty}^{0} \left( \partial_t \varphi(u,y)+\frac{1}{2} \partial^2_y \varphi(u,y)+\phi''(u)y\varphi(u,y) \right) g(u,y) dydu=0
\end{align*}
Thus we get the PDE in the distribution sense, but also in the classical sense because $g$ is smooth on the interior of its domain. 
\end{proof}
We give now an explicit formula for the functional $\Phi$ that was introduced in the previous section.
\begin{prop}
The function $\Phi$ can be expressed as 
\begin{align*}
\Phi(s,x,t)=\frac{-x}{\sqrt{2\pi(t-s)^3}} e^{-\frac{x^2}{2(t-s)}} \text{exp} \left( \phi'(s)x-\frac{1}{2} \int_{s}^{t} (\phi'(u))^2 du \right) \times\\
\mathbb{E}^{(s,-x) \rightarrow (t,0)}\left[\text{exp} \left (-\int_{s}^{t} \phi''(u)B^{br}(u)du \right)\right]
\end{align*}
for $s < t$ and $x<0$. $B^{br}$ here is a three-dimensional Bessel bridge from $(s,-x)$ to $(t,0)$. 
\end{prop}
\begin{proof}
From Lemma \ref{joint}, we have that 
\begin{align*}
\Phi(s,x,t)=-\frac{1}{2} \partial_y f(s,x,t,0)
\end{align*}
Since
\begin{align*}
f(s,x,t,y)=\text{exp} \left (-\phi'(t)y+\phi'(s)x-\frac{1}{2} \int_{s}^{t} (\phi'(u))^2 du \right )g(s,x,t,y)
\end{align*}
and
\begin{align*}
\begin{split}
\partial_y g(s,x,t,0)=\lim_{y \uparrow 0} \partial_y G(s,x,t,y) \mathbb{E}\left[\text{exp} \left (-\int_{s}^{t} \phi''(u)B(u)du \right)\vert B(s)=-x, B(t)=-y\right]+\\
\lim_{y \uparrow 0} G(s,x,t,y) \partial_y \mathbb{E}\left[\text{exp} \left (-\int_{s}^{t} \phi''(u)B(u)du \right)\vert B(s)=-x, B(t)=-y\right]
\end{split}
\end{align*}
it suffices to prove that 
\begin{align*}
\lim_{y \uparrow 0} \partial_y \mathbb{E}\left[\text{exp} \left (-\int_{s}^{t} \phi''(u)B(u)du \right)\vert B(s)=-x, B(t)=-y\right] < \infty
\end{align*}
as $G(s,x,t,0)=0$. We have by Hopital's rule applied twice
\begin{align*}
\lim_{y \uparrow 0} \partial_y \mathbb{E}\left[\text{exp} \left (-\int_{s}^{t} \phi''(u)B(u)du \right)\vert B(s)=-x, B(t)=-y\right]&=\lim_{y \uparrow 0} \frac{(\partial_y g)G-(\partial_y G)g}{G^2}\\
&=\lim_{y \uparrow 0} \frac{(\partial^2_{y}g)G-(\partial_y^2 G)g}{2 G \partial_y G}\\
&=\lim_{y \uparrow 0} \frac{\partial_y^2 g}{2\partial_y G} -\lim_{y \uparrow 0} \frac{(\partial_y ^2 G) g}{2G \partial_y G}\\
&=-\lim_{y \uparrow 0} \frac{(\partial_y^2G)g}{2G \partial_y G}\\
&=-\lim_{y \uparrow 0} \frac{(\partial_y ^3 G)g +(\partial_y ^2 G)\partial_y g}{2(\partial_y G)^2+2G \partial_y ^2 G}\\
\lim_{y \uparrow 0} \partial_y \mathbb{E}\left[\text{exp} \left (-\int_{s}^{t} \phi''(u)B(u)du \right)\vert B(s)=-x, B(t)=-y\right]&=0
\end{align*}
In the fourth line we used the fact that $\lim_{y \uparrow 0} \partial_y^2 g=0$. This follows from the PDE \eqref{pde-g} verified by $g$ and the fact that $g(s,x,t,0)=0$. Moreover because $\lim_{y \uparrow 0} \partial_y G \ne 0$, we can conclude that the limit is equal to zero in the penultimate equality.\\

To finish the proof, we refer to the fact that the weak limit of the law of the three-dimensional Bessel process conditioned to end at $y$ when $y$ goes to zero is that of the corresponding three-dimensional Bessel bridge, and thus the result follows from the expression of the Green function $G$. 
\end{proof}
We are ready to give an explicit formula of the kernel $K^{\phi}$.
\begin{prop}\label{kernel}
The kernel $K^{\phi}$ has the following expression 
\begin{align*}
K^{\phi}(y,z)=\frac{\phi'(z)-\phi'(y)}{\sqrt{2\pi (z-y)^3}} \text{exp} \left (-\frac{1}{2} \int_{y}^{z} (\phi'(u))^2 du \right ) 
\mathbb{E}\left[\text{exp} \left (-\int_{y}^{z} \phi''(u) \textbf{e}(u)du \right)\right]
\end{align*}
for $y\le z$, where $(\textbf{e}(u) ,   y \le u \le z)$ is a Brownian excursion on $[y,z]$. 
\end{prop}
\begin{proof}
Recall that $K^{\phi}$ is given by 
\begin{align*}
K^{\phi}(y,z)=2\int_{y}^{z} \int_{0}^{\infty} \phi''(\omega)\Phi(\omega,-x,z)\tilde{\Phi}(-\omega,-x,-y)dxd\omega
\end{align*}
Remember that $\tilde{\Phi}$ is the same as $\Phi$ with the function $\phi$ replaced by $\phi(-\cdot)$. Hence
\begin{align*}
\Phi(\omega,-x,z)\tilde{\Phi}(-\omega,-x,-y)=\frac{x^2}{2\pi\sqrt{(z-\omega)^3(\omega-y)^3}}e^{-\frac{x^2}{2(z-\omega)}-\frac{x^2}{2(\omega-y)}}\times \\
\text{exp} \left (-\frac{1}{2} \int_{\omega}^{z} (\phi'(u))^2du -\frac{1}{2} \int_{-\omega}^{-y} (\phi'(-u))^2 du \right)\times \\
\mathbb{E}^{(\omega,x) \rightarrow (z,0)}\left[\text{exp} \left (-\int_{\omega}^{z} \phi''(u)B^{br}(u)du \right)\right]\times \\
\mathbb{E}^{(-\omega,x) \rightarrow (-y,0)}\left[\text{exp} \left (-\int_{-\omega}^{-y} \phi''(-u)B^{br}(u)du \right)\right]
\end{align*}
Consider now a Brownian excursion $\textbf{e}$ on $[y,z]$, conditionally on its value at $\omega \in [y,z] $, the two paths $(\textbf{e}(u) , ~~y \le u \le  \omega)$ and $(\textbf{e}(u), ~~ \omega \le u \le z)$ are independent, and each path has the distribution of a three-dimensional Bessel bridge. Furthermore, because of the Brownian scaling we have that 
\begin{equation}\label{scaling}
(\textbf{e}(u) , y \le u \le z) \stackrel{d}{=} (\sqrt{y-z}\textbf{e}^{\text{std}}\left(\frac{u-y}{z-y}\right) , y \le u \le z)
\end{equation}
where $(\textbf{e}^{\text{std}}(u) , 0 \le u \le 1)$ is a standard Brownian excursion. Thus, using the fact that 
\begin{align*}
\mathbb{P}\left[\textbf{e}^{\text{std}}(t) \in dx\right]=\frac{2x^2}{\sqrt{2\pi t^3(1-t)^3}}e^{-\frac{x^2}{2t(1-t)}} dx
\end{align*}
then it follows that for $\omega \in [y,z]$ 
\begin{align*}
\mathbb{P}\left[\textbf{e}(\omega) \in dx\right]=\frac{2x^2 \sqrt{(z-y)^3}}{\sqrt{2\pi (z-\omega)^3(\omega-y)^3}}e^{-\frac{x^2}{2(z-\omega)}-\frac{x^2}{2(\omega-y)}} dx
\end{align*}
Thus by the time-reversal property of the three-dimensional Bessel bridges we have that 
\begin{align*}
\Phi(\omega,-x,z)\tilde{\Phi}(-\omega,-x,-y)=\frac{1}{\sqrt{2\pi(z-y)^3}}
 \mathbb{E}[\text{exp} \left(-\int_{y}^{z} \phi''(u)\textbf{e}(u)du \right)|\textbf{e}(\omega)=x]\times \\
\frac{\mathbb{P}[\textbf{e}(\omega)\in dx]}{dx}
\end{align*}
By integrating with respect to $x$ and $\omega$ we get the desired result.
\end{proof}
The next theorem gives a closed formula for the function $j$. 
\begin{thm}\label{jou}
Let $s \in \mathbb{R}$, define the function $l^s$ on $(0,\infty)$ by 
\begin{align*}
l^s(u)=\text{exp} \left (-\frac{1}{2} \int_{s}^{s+u} \phi'(z)^2 dz \right) \mathbb{E}\left[\text{exp} \left ( - \int_{s}^{s+u} \phi''(z)\textbf{e}(z)dz \right)\right],~~u>0
\end{align*}
where $\textbf{e}$ is a Brownian excursion on $[s,s+u]$. Then 
\begin{align*}
j(s)=-\phi'(s)+\int_{0}^{\infty} \frac{l^{s}(u)-1}{\sqrt{2\pi u^3}}du
\end{align*}
\end{thm}
\begin{proof}
The function $J$ is defined as 
\begin{align*}
J(s,x)&=\mathbb{P}[S(\omega)<0 \text{ for all } \omega \ge s \vert S(s)=x]\\
&=\lim_{t \rightarrow \infty} \mathbb{P}[S(\omega) < 0 \text{ for all } s \le \omega \le t \vert S(s)=x]\\
&=\lim_{t \rightarrow \infty} \int_{-\infty}^{0} f(s,x,t,y)dy\\
&=e^{\phi'(s)x} \lim_{t \rightarrow \infty} e^{-\frac{1}{2} \int_{s}^t (\phi'(u))^2du} \int_{-\infty}^0 e^{-\phi'(t)y}g(s,x,t,y)dy\\
&=e^{\phi'(s)x} \lim_{t \rightarrow \infty} e^{-\frac{1}{2} \int_{s}^t (\phi'(u))^2du} \int_{0}^{\infty} e^{\phi'(t)y}m(s,x,t,y)dy
\end{align*}
where the function $m$ is defined as
\begin{align*}
m(s,x,t,y)=g(s,x,t,-y)
\end{align*}
It verifies the following PDE 
\begin{equation}\label{pde-m}
\partial_t m=\frac{1}{2} \partial^2_{yy}m-\phi''(t)ym
\end{equation}
Because of the asymptotic behavior of $g$ in space at infinity, we can define for every $\lambda \in \mathbb{R}$ the Laplace transform 
\begin{align*}
\hat{m}(s,x,t,\lambda)=\int_{0}^{\infty} e^{\lambda y}m(s,x,t,y)dy
\end{align*}
From the representation formula of the function $h$ (and thus that of $g$) in the statement of Theorem \ref{representation} and the fast decay of the Green functions $G$ and $\hat{G}$ in space, we can interchange the order of differentiation and integration for the Laplace transform $\hat{m}$, hence 
\begin{align*}
\partial_t \hat{m}&=\int_{0}^{\infty} e^{\lambda y} \partial_t m(s,x,t,y)dy\\
&=\int_{0}^{\infty} e^{\lambda y} \left(\frac{1}{2} \partial^2_{yy} m(s,x,t,y) - \phi''(t) y m(s,x,t,y)\right)dy\\
&=\frac{1}{2}\left[e^{\lambda y} \partial_y m(s,x,t,y)\right]_{0}^{\infty}+\frac{1}{2} \lambda^2 \hat{m}(s,x,t,\lambda)-\phi''(t)\partial_{\lambda} \hat{m}(s,x,t,\lambda)\\
&=\frac{1}{2} \lambda^2 \hat{m}(s,x,t,\lambda) -\phi''(t)\partial_{\lambda} \hat{m}(s,x,t,\lambda)-\frac{1}{2} \partial_y m(s,x,t,0)
\end{align*}
by integration by parts and using the fact that $m(s,x,t,0)=0$. From the expression of $g$ we deduce that
\begin{align*}
\partial_y m(s,x,t,0)=-\partial_y g(s,x,t,0)=\frac{-2x}{\sqrt{2\pi (t-s)^3}}e^{-\frac{x^2}{2(t-s)}}\times \\
\mathbb{E}\left[\text{exp} \left (-\int_{s}^{t} \phi''(u)B(u)du \right ) \vert B(s)=-x,B(t)=0 \right]\\
=2\Phi(s,x,t)\text{exp} \left (-\phi'(s)x + \frac{1}{2} \int_{s}^{t} (\phi'(u))^2 du \right )=:-2\Upsilon(t)
\end{align*}
Since $x$ and $s$ are fixed for now, we will often omit them when writing out expressions where they do not vary. Thus, the PDE verified by $\hat{m}$ takes the form
\begin{align*}
\partial_t \hat{m} +\phi''(t)\partial_{\lambda} \hat{m} -\frac{1}{2} \lambda^2 \hat{m} - \Upsilon(t)=0
\end{align*}
This is a first order non-linear PDE that can be solved by the method of characteristics. If we denote the variables by $x_1:=t$ and $x_2:=\lambda$ and the value of the function $z=\hat{m}(x_1,x_2)$, the characteristic ODEs take the form
\begin{center}
$
\left\{
    \begin{array}{ll}
        \dot{x_1}(u)=1 \\
        \dot{x_2}(u)=\phi''(x_1(u)) \\
        \dot{z}(u)=\frac{1}{2}x_2^2(u)z(u)+\Upsilon(x_1(u))
    \end{array}
\right.
$
\end{center}
We choose the initial conditions such that $x_1(u)=u$ and $x_2(u)=\phi'(u)+(\lambda-\phi'(t))$ for $u \ge s$. Hence
\begin{equation*}
\dot{z}(u)=\frac{1}{2}(\phi'(u)+\lambda-\phi'(t))^2z(u)+\Upsilon(u)
\end{equation*}
Introduce the function $v^{\lambda}$ defined by
\begin{equation*}
v^{\lambda}(u)=\text{exp}\left (-\frac{1}{2} \int_{s}^{u} (\phi'(z)+\lambda-\phi'(t))^2dz\right)
\end{equation*}
Then it is clear that
\begin{equation*}
(\dot{v^{\lambda}z})(u)=v^{\lambda}(u)\Upsilon(u)
\end{equation*}
In order to avoid the singularity at $\{t=s\}$, we integrate thus between $s+\epsilon$ and $t$ for $\epsilon>0$ to get that 
\begin{equation*}
v^{\lambda}(t)z(t)-v^{\lambda}(s+\epsilon)z(s+\epsilon)=\int_{s+\epsilon}^{t} v^{\lambda}(u)\Upsilon(u)du
\end{equation*}
which is equivalent to
\begin{align*}
\hat{m}(s,x,t,\lambda)v^{\lambda}(t)-\hat{m}(s,x,s+\epsilon,\phi'(s+\epsilon)+\lambda-\phi'(t))v^{\lambda}(s+\epsilon)=
\int_{s+\epsilon}^{t} v^{\lambda}(u)\Upsilon(u)du
\end{align*}
By taking $\lambda=\phi'(t)$, we get
\begin{equation}\label{charac-express}
\begin{split}
\hat{m}(s,x,t,\phi'(t))e^{-\frac{1}{2}\int_{s}^{t} \phi'(u)^2 du}-\hat{m}(s,x,s+\epsilon,\phi'(s+\epsilon))e^{-\frac{1}{2}\int_{s}^{s+\epsilon} \phi'(u)^2 du}\\=
\int_{s+\epsilon}^{t} e^{-\frac{1}{2}\int_{s}^{u} \phi'(\omega)^2 d\omega}\Upsilon(u)du
\end{split}
\end{equation}
As $J(s,x)=e^{\phi'(s)x} \lim\limits_{t\rightarrow \infty} e^{-\frac{1}{2} \int_{s}^{t} (\phi'(u))^2du} \hat{m}(s,x,t,\phi'(t))$. By sending $t$ to $\infty$ in the expression \eqref{charac-express}, we have
\begin{align*}
J(s,x)=e^{\phi'(s) x} \left [\hat{m}(s,x,s+\epsilon,\phi'(s+\epsilon))e^{-\frac{1}{2}\int_{s}^{s+\epsilon} \phi'(u)^2 du} \right.+\\
\left. \int_{s+\epsilon}^{\infty} e^{-\frac{1}{2}\int_{s}^{u} \phi'(\omega)^2 d\omega}\Upsilon(s,x,u)du \right]
\end{align*}
It follows that
\begin{align}\label{terms}
\begin{split}
j(s):=\lim_{x \uparrow 0} \frac{\partial}{\partial x} J(s,x)=e^{-\frac{1}{2}\int_{s}^{s+\epsilon} \phi'(u)^2 du} \lim_{x \uparrow 0} \frac{\partial}{\partial x}\hat{m}(s,x,s+\epsilon,\phi'(s+\epsilon))+\\
\int_{s+\epsilon}^{\infty} e^{-\frac{1}{2}\int_{s}^{u} \phi'(\omega)^2 d\omega} \lim_{x\uparrow 0}\frac{\partial}{\partial x} \Upsilon(s,x,u)du
\end{split}
\end{align}
since $m(s,0,s+\epsilon,\cdot)=0$, and we can interchange differentiation and the integral sign in the second term because we are away from the singularity line $\{t=s\}$. Now, we have that
\begin{align*}
\hat{m}(s,x,s+\epsilon,\phi'(s+\epsilon))=\int_{0}^{\infty} e^{\phi'(s+\epsilon)y}m(s,x,s+\epsilon,y)dy
\end{align*}
It is clear that $m$ is smooth in the parameters $(s,x)$ as well. Our analysis of regularity of the function $f(s,x,t,y)$ consisted of using the Kolmogorov forward equation where the parameters were $t$ and $y$, but similarly the Kolmogorov backward equation that holds for the parameters $s$ and $x$, we see that the solution enjoys the same smoothness and integrability properties away from the line $\{s=t\}$ (it is formally just the adjoint problem). Hence we can differentiate inside the integral sign to get
\begin{align*}
\lim_{x \uparrow 0} \frac{\partial}{\partial x} \hat{m}(s,x,s+\epsilon,\phi'(s+\epsilon))&=\int_{0}^{\infty} e^{\phi'(s+\epsilon)y} \lim_{x \uparrow 0} \frac{\partial}{\partial x} m(s,x,s+\epsilon,y)dy
\end{align*}
since we have that
\begin{align*}
\lim_{x \uparrow 0} \frac{\partial}{\partial x} m(s,x,s+\epsilon,y)=-\frac{2y}{\sqrt{2\pi \epsilon^3}} e^{-\frac{y^2}{2\epsilon}}\times \\ \mathbb{E}\left[\text{exp} \left (-\int_{s}^{s+\epsilon} \phi''(u)B(u)du \right ) \vert B(s)=0,B(s+\epsilon)=y\right]
\end{align*}
Thus
\begin{align*}
\lim_{x \uparrow 0} \frac{\partial}{\partial x} \hat{m}(s,x,s+\epsilon,\phi'(s+\epsilon))=-\int_{0}^{\infty} e^{\phi'(s+\epsilon)y-\frac{y^2}{2\epsilon}}\frac{2y}{\sqrt{2\pi \epsilon^3}}\times \\
\mathbb{E}\left[\text{exp} \left (-\int_{s}^{s+\epsilon} \phi''(u)B(u)du \right ) \vert B(s)=0,B(s+\epsilon)=y\right]dy
\end{align*}
However, the density of a three-dimensional Bessel process is given by 
\begin{equation}\label{density}
\mathbb{P}[B(s+\epsilon) \in dy \vert B(s)=0]=\frac{2y^2}{\sqrt{2\pi \epsilon^3}} e^{-\frac{y^2}{2\epsilon}} dy
\end{equation}
Hence
\begin{align*}
\lim_{x \uparrow 0} \frac{\partial}{\partial x} \hat{m}(s,x,s+\epsilon,\phi'(s+\epsilon))=\\
-\mathbb{E}\left[\frac{1}{B(s+\epsilon)}\text{exp} \left(\phi'(s+\epsilon)B(s+\epsilon)-\int_{s}^{s+\epsilon} \phi''(u)B(u)du \right) \vert B(s)=0\right]\\
=-\mathbb{E}\left[\frac{1}{B(\epsilon)}\text{exp} \left(\phi'(s+\epsilon)B(\epsilon)-\int_{0}^{\epsilon} \phi''(u+s)B(u)du \right) \vert B(0)=0\right]
\end{align*}
However by Brownian scaling, we know that 
\begin{align*}
(B(u) , u \ge 0) \overset{\mathrm{d}}{=} (\sqrt{\epsilon} B\left(\frac{u}{\epsilon}\right) , u \ge 0)
\end{align*}
Hence
\begin{align*}
\lim_{x \uparrow 0} \frac{\partial}{\partial x} \hat{m}(s,x,s+\epsilon,\phi'(s+\epsilon))&=
-\mathbb{E}\left[\frac{1}{\sqrt{\epsilon}B(1)}\text{exp} \left(\phi'(s+\epsilon)\sqrt{\epsilon}B(1)-\right. \right.\\
\left. \left. \sqrt{\epsilon^3}\int_{0}^{1} \phi''(\epsilon u+s)B(u)du \right) \vert B(0)=0\right]\\
&=-\frac{1}{\sqrt{\epsilon}}\mathbb{E}\left[\frac{1}{B(1)}\right]+\phi'(s)+O(\sqrt{\epsilon})
\end{align*}
It follows then that
\begin{equation}\label{express1}
\lim_{x \uparrow 0} \frac{\partial}{\partial x} \hat{m}(s,x,s+\epsilon,\phi'(s+\epsilon))=-\frac{2}{\sqrt{2\pi \epsilon}}+\phi'(s)+O(\sqrt{\epsilon})
\end{equation}
for $\epsilon$ small. The expectation of the inverse of $B(1)$ is computed using the density given in \eqref{density}. Now, on the other hand for the second term in \eqref{terms}, we have 
\begin{align*}
\lim_{x\uparrow 0} \frac{\partial}{\partial x} \Upsilon(s,x,u)&=-\partial_x \Phi(s,0,u) \text{exp} \left(\frac{1}{2} \int_{s}^{u}\phi'(\omega)^2 d\omega\right)\\
&=\frac{1}{\sqrt{2\pi(u-s)^3}} \mathbb{E}\left[\text{exp} \left (-\int_{s}^{u} \phi''(z)\textbf{e}(z)dz \right)\right] \\
\end{align*}
Hence
\begin{equation}\label{express2}
\int_{s+\epsilon}^{\infty} e^{-\frac{1}{2}\int_{s}^{u} \phi'(\omega)^2 d\omega} \lim_{x\uparrow 0}\frac{\partial}{\partial x} \Upsilon(s,x,u)du=\int_{\epsilon}^{\infty} \frac{l^s(u)}{\sqrt{2\pi u^3}}du
\end{equation}
and thus, from combining \eqref{terms}, \eqref{express1} and \eqref{express2} we get
\begin{align*}
j(s)=\int_{\epsilon}^{\infty} \frac{l^s(u)}{\sqrt{2\pi u^3}}du-\frac{2}{\sqrt{2\pi \epsilon}}-\phi'(s)+O(\sqrt{\epsilon})
\end{align*}
Finally, see that 
\begin{align*}
\int_{\epsilon}^{\infty} \frac{du}{\sqrt{2\pi u^3}} =\frac{2}{\sqrt{2\pi \epsilon}}
\end{align*}
and then send $\epsilon$ to zero to finish the proof.
\end{proof}
\begin{remark}
When $\phi$ is parabolic ($\phi(y)=y^2$), the term $\phi''$ in the PDE \eqref{pde-m} of $m$ becomes a constant and thus it takes the simple form
\begin{align*}
\partial_t m=\frac{1}{2} \partial^2_{yy} m -2ym
\end{align*}
By taking the Fourier transform in \textit{time} we get 
\begin{align*}
\frac{1}{2} (\hat{m}(\tau,y))''=(i\tau+2y) \hat{m}(\tau,y)
\end{align*}
This is a Sturm-Liouville equation. Its solution can be expressed in terms of Airy functions, from which follows all the analytical descriptions that Groeneboom found in \cite{groeneboom89}. It is clear that when $\phi''$ is not constant, this method fails which makes the study more delicate as one doesn't have any asymptotic or regularity properties of the function $m$, which was a crucial part in the analysis of Groeneboom. For those reasons, we had to take advantage of the \textit{space} Laplace transform. 
\end{remark}
As a consequence of the explicit formula of $j$ and $\Phi$, we are able to provide the joint distribution of the maximum of the process $(W(\omega)-\phi(\omega))_{\omega \ge s}$ and its location. This is given by the expression of $\Phi$ and $j$ and using Lemma \ref{joint}. However, the formula is involving many terms, in particular the Bessel bridge area. On the other hand, the density of the location of the maximum takes a simpler formula. This is a generalization of Chernoff distribution, where the parabolic drift is replaced by any strictly convex drift $\phi$.
 
\begin{thm}
Let $\omega_M$ be the location of the unique maximum of the process $(S(\omega)=W(\omega)-\phi(\omega))_{\omega \in \mathbb{R}}$, its density is equal to 
\begin{align*}
\frac{\mathbb{P}[\omega_M \in dt]}{dt}=\frac{1}{2}j(t)\tilde{j}(-t) 
\end{align*}
where $\tilde{j}$ is the analogue of $j$ for the process $\tilde{S}(\omega):=W(\omega)-\phi(-\omega)$.
\end{thm}
\begin{proof}
We will prove the equality for $t \ge 0$, the case $t \le 0$ is completely identical. From Lemma \ref{joint} with $s=0$ and any $x>z$ 
\begin{align*}
\frac{\mathbb{P}[\text{argmax}_{\omega \ge 0} S(\omega) \in dt, \max\limits_{\omega \ge 0} S(\omega) \in dz \vert S(0)=x]}{dtdz}=\frac{1}{2}j(t)\partial_y f(0,x-z,t,0)
\end{align*}
Hence
\begin{align*}
\mathbb{P}[\omega_M \in dt \vert S(0)=x]=\int_{x}^{+\infty} \mathbb{P}[\text{argmax}_{\omega \ge 0} S(\omega) \in dt, \max\limits_{\omega \ge 0} S(\omega) \in dz,\\
 \max\limits_{ \omega \le 0} S(\omega) < z \vert S(0)=x]\\
=\int_{x}^{+\infty}\frac{1}{2}j(t)\partial_y f(0,x-z,t,0)\mathbb{P}[S(\omega)<z \text{ for all } \omega \le 0 \vert S(0)=x]dzdt
\end{align*}
by independence of the paths $(S(\omega), \omega \le 0)$ and $(S(\omega), \omega \ge 0)$. However by time reversal of the Brownian motion we have
\begin{align*}
\mathbb{P}[S(\omega)<z \text{ for all } \omega \le 0 \vert S(0)=x]&=\mathbb{P}[\tilde{S}(\omega)<z \text{ for all } \omega \ge 0 \vert \tilde{S}(0)=x]\\
&=\mathbb{P}[\tilde{S}(\omega)<0 \text{ for all } \omega \ge 0 \vert \tilde{S}(0)=x-z]\\
&=\tilde{J}(0,x-z)
\end{align*}
Thus 
\begin{align*}
\frac{\mathbb{P}[\omega_M \in dt \vert S(0)=x]}{dt}=\int_{-\infty}^{0} \frac{1}{2}j(t)\partial_y f(0,z,t,0)\tilde{J}(0,z)dz
\end{align*}
Notice that the right hand-side is independent of $x$, so we can drop the conditional probability in the left hand-side. Moreover by \eqref{backward}, we have 
\begin{equation}\label{backward-1}
\partial_y f(0,z,t,0)=\partial_x \tilde{f}(-t,0,0,z)
\end{equation}
Using the expression of the entrance law of the process $\tilde{S}^{\downarrow}$ in \eqref{entrance-1}, we have 
\begin{equation}\label{entrance-2}
\mathbb{P}[\tilde{S}^{\downarrow}(0) \in dz \vert \tilde{S}^{\downarrow}(-t)=0]=\frac{\tilde{J}(0,z)}{\tilde{j}(-t)}\partial_x \tilde{f}(-t,0,0,z)dz
\end{equation}
Hence combining  \eqref{backward-1} and \eqref{entrance-2} we get 
\begin{align*}
\int_{0}^{\infty} \partial_y f(0,z,t,0)\tilde{J}(0,z)dz=\tilde{j}(-t)\int_{-\infty}^{0} \mathbb{P}[\tilde{S}^{\downarrow}(0) \in dz \vert \tilde{S}^{\downarrow}(-t)=0]=\tilde{j}(-t)
\end{align*}
which completes the proof.

\end{proof}
\begin{remark}
This last theorem is exactly Theorem \ref{chernoff} by noticing that $f^{\phi}(t)=-j(t)$ and $f^{\phi(-\cdot)}(-t)=-\tilde{j}(-t)$.
\end{remark}
\begin{remark}
From \cite{groeneboom89} results in the parabolic drift case, the Chernoff distribution can be expressed as 
\begin{align*}
\frac{\mathbb{P}[\text{argmax}_{z \in \mathbb{R}}(W(z)-z^2) \in dt]}{dt}=\frac{1}{2} k(t)k(-t)
\end{align*}
where $k(t)=e^{\frac{2}{3}t^3}g(t)$ and $g$ has the Fourier transform given by 
\begin{align*}
\hat{g}(\tau):=\int_{-\infty}^{\infty} e^{it\tau}g(t)dt=\frac{2^{\frac{1}{3}}}{\mathrm{Ai}(i2^{-\frac{1}{3}}\tau)}
\end{align*}
This expression is not clear from the formula we provided in Theorem \ref{chernoff}. We will prove thus in the following proposition that those two indeed coincide.
\end{remark}
\begin{prop}
For any $t \in \mathbb{R}$ we have 
\begin{align*}
2t+\int_{0}^{\infty} \frac{1}{\sqrt{2\pi u^3}}\left(1-e^{-\frac{2}{3}((u+t)^3-t^3)}\mathbb{E}\left[\text{exp} \left(-2\int_{0}^{u} \textbf{e}(z)dz \right)\right]\right)du=\\
\frac{e^{\frac{2}{3}t^3}}{2\pi} \int_{-\infty}^{\infty} e^{-itv} \hat{g}(v)dv
\end{align*}
\end{prop}
\begin{proof}
From equation (1.6) in \cite{chernoff} \footnote{There is a typo in the published paper, the term $4^{\frac{2}{3}}$ in the denominator should be there instead of $4^{\frac{1}{3}}$.} , we have that 
\begin{align}\label{check1}
\begin{split}
\frac{1}{2\pi} \int_{v=-\infty}^{\infty}\frac{\mathrm{Ai}(i\xi-4^{\frac{1}{3}}x)}{\mathrm{Ai}(i\xi)} \int_{u=0}^{\infty} e^{iuv-\frac{2}{3}((u+t)^3-t^3)}dudv=e^{-2tx}\\-\frac{e^{\frac{2}{3}t^3}}{4^{\frac{2}{3}}}\int_{v=-\infty}^{\infty} e^{-itv}\frac{\mathrm{Ai}(i\xi)\mathrm{Bi}(i\xi-4^{\frac{1}{3}}x)-\mathrm{Ai}(i\xi-4^{\frac{1}{3}}x)\mathrm{Bi}(i\xi)}{\mathrm{Ai}(i\xi)}dv
\end{split}
\end{align}
where $\xi=2^{-\frac{1}{3}}v$, and $\mathrm{Bi}$ is the second Airy function. By differentiating both sides with respect to $x$ and sending $x$ to zero, we get
\begin{equation}\label{check}
\frac{e^{\frac{2}{3}t^3}}{4^{\frac{1}{3}}\pi}\int_{v=-\infty}^{\infty}\frac{e^{-itv}}{\mathrm{Ai}(i\xi)} dv=\\
2t+\lim_{x \uparrow 0}\frac{\partial}{\partial x} \frac{1}{2\pi}\int_{v=-\infty}^{\infty}\frac{\mathrm{Ai}(i\xi-4^{\frac{1}{3}}x)}{\mathrm{Ai}(i\xi)} \int_{u=0}^{\infty} e^{iuv-\frac{2}{3}((u+t)^3-t^3)}dudv
\end{equation}
as the Wronskian of the Airy functions $\mathrm{Ai}$ and $\mathrm{Bi}$ is constant and equal to $\frac{1}{\pi}$. In the right-hand side of \eqref{check1}, we cannot differentiate inside the integral sign because it becomes divergent. However for fixed $x<0$, the integrand is absolutely integrable and thus we can use Fubini theorem. Now from \cite{janson}[Equation 384, Page 141] we have that 
\begin{align*}
-\int_{0}^{\infty}e^{-\lambda s}\mathbb{E}\left[\text{exp}\left(-2 \int_{0}^{s} B(u)du \right) \vert B(s)=-x\right]\frac{x}{\sqrt{2\pi s^3}} e^{-\frac{x^2}{2s}}ds=
\frac{\mathrm{Ai}(2^{-\frac{1}{3}}\lambda -4^{\frac{1}{3}}x)}{\mathrm{Ai}(2^{-\frac{1}{3}}\lambda)}
\end{align*}
where $B$ is as usual a three-dimensional Bessel process. Thus, by inverse Laplace transform we have 
\begin{align*}
-\mathbb{E}\left[\text{exp}\left(-2 \int_{0}^{u} B(z)dz \right) \vert B(u)=-x\right]\frac{x}{\sqrt{2\pi u^3}} e^{-\frac{x^2}{2u}}=
\frac{1}{2\pi} \int_{-\infty}^{\infty} e^{iuv} \frac{\mathrm{Ai}(i\xi-4^{\frac{1}{3}}x)}{\mathrm{Ai}(i\xi)}dv
\end{align*}
Hence the integral in the RHS of \eqref{check} is equal to 
\begin{equation}\label{check2}
-\int_{0}^{\infty} e^{-\frac{2}{3}((u+t)^3-t^3)} \frac{x}{\sqrt{2\pi u^3}} e^{-\frac{x^2}{2u}}\mathbb{E}\left[\text{exp}\left(-2 \int_{0}^{u} B(z)dz \right) \vert B(u)=-x\right]du
\end{equation}
By splitting this integral on $(0,\epsilon)$ and $(\epsilon,\infty)$, we can interchange the integral and the differentiation for the integral on $(\epsilon,\infty)$, and so we get after sending $x$ to zero 
\begin{equation}\label{greater-epsilon}
-\int_{\epsilon}^{\infty} e^{-\frac{2}{3}((u+t)^3-t^3)} \frac{1}{\sqrt{2\pi u^3}} \mathbb{E}\left[\text{exp} \left(-2\int_{0}^{u} \textbf{e}(z)dz \right)\right]du
\end{equation} 
where $\textbf{e}$ is as usual a Brownian excursion on the corresponding interval. As for the first term (the integral on $(0,\epsilon)$), by the change of variable $y=\frac{x}{\sqrt{u}}$ ($dy=-\frac{x}{2\sqrt{u^3}}du$), it is equal to
\begin{align*}
-\int_{0}^{\epsilon} e^{-\frac{2}{3}((u+t)^3-t^3)} \frac{x}{\sqrt{2\pi u^3}} e^{-\frac{x^2}{2u}}\mathbb{E}\left[\text{exp}\left(-2 \int_{0}^{u} B(z)dz \right) \vert B(u)=-x\right]du\\
=\int_{-\infty}^{\frac{x}{\sqrt{\epsilon}}} e^{-\frac{2}{3}((\frac{x^2}{y^2}+t)^3-t^3)} \frac{2}{\sqrt{2\pi}} e^{-\frac{y^2}{2}}\mathbb{E}\left[\text{exp}\left(-2\frac{x^3}{y^3} \int_{0}^{1} B(z)dz \right) \vert B(1)=-y\right]dy
\end{align*}
by Brownian scaling on the Bessel process $B$. Differentiating with respect to $x$, we get by Leibniz rule
\begin{equation}\label{less-epsilon}
\frac{2}{\sqrt{2\pi \epsilon}}e^{-\frac{2}{3}((\epsilon+t)^3-t^3)}e^{-\frac{x^2}{2\epsilon}}\mathbb{E}\left[\text{exp}\left(-2\sqrt{\epsilon^3} \int_{0}^1 B(z)dz \right) \vert B(1)=-\frac{x}{\sqrt{\epsilon}}\right]+F^{\epsilon}(x)
\end{equation}
where $F^{\epsilon}$ is equal to 
\begin{align*}
F^{\epsilon}(x)=\frac{2}{\sqrt{2\pi}} \int_{-\infty}^{\frac{x}{\sqrt{\epsilon}}} \left(-4\frac{x^5}{y^6}-8t\frac{x^3}{y^4}-4t^2\frac{x}{y^2}-6\frac{x^2}{y^3}\right)e^{-\frac{y^2}{2}}e^{-\frac{2}{3}((\frac{x^2}{y^2}+t)^3-t^3)}\times\\
\mathbb{E}\left[\text{exp}\left(-2\frac{x^3}{y^3} \int_{0}^{1} B(z)dz \right) \vert B(1)=-y\right]dy
\end{align*}
However we have that for $x$ small enough (such that $\vert \frac{x}{\sqrt{\epsilon}} \vert=-\frac{x}{\sqrt{\epsilon}} \le 1$)
\begin{align*}
\left\vert \int_{-\infty}^{\frac{x}{\sqrt{\epsilon}}} \frac{x}{y^2} e^{-\frac{y^2}{2}}e^{-\frac{2}{3}((\frac{x^2}{y^2}+t)^3-t^3)} \mathbb{E} \left[ \text{exp} \left (-2\frac{x^3}{y^3} \int_{0}^{1} B(z)dz \right) \vert B(1)=-y\right]dy \right \vert \le \\
\vert x \vert \int_{-\frac{x}{\sqrt{\epsilon}}}^{\infty} \frac{e^{-\frac{y^2}{2}}}{y^2}dy \le \vert x \vert (1-\frac{\sqrt{\epsilon}}{x}+\int_{1}^{\infty} e^{-\frac{y^2}{2}}dy)
\end{align*}
so 
\begin{align*}
\limsup\limits_{x \uparrow 0} \left\vert \int_{-\infty}^{\frac{x}{\sqrt{\epsilon}}} \frac{x}{y^2} e^{-\frac{y^2}{2}}e^{-\frac{2}{3}((\frac{x^2}{y^2}+t)^3-t^3)} \mathbb{E} \left[ \text{exp} \left (-2\frac{x^3}{y^3} \int_{0}^{1} B(z)dz \right) \vert B(1)=-y\right]dy \right \vert \le \sqrt{\epsilon}
\end{align*}
Similarly with the other terms we find that there is a constant $C>0$ (that depends on $t$) such that 
\begin{align*}
\limsup\limits_{x \uparrow 0} \vert F^{\epsilon}(x) \vert \le C \sqrt{\epsilon}
\end{align*}
Hence, by combining \eqref{greater-epsilon} and \eqref{less-epsilon}, the limit of the derivative of the expression in \eqref{check2} when $x$ goes to zero is equal to
\begin{align*}
-\int_{\epsilon}^{\infty} e^{-\frac{2}{3}((u+t)^3-t^3)} \frac{1}{\sqrt{2\pi u^3}} \mathbb{E}\left[\text{exp} \left(-2\int_{0}^{u} \textbf{e}(z)dz \right)\right]du+\\
\frac{2}{\sqrt{2\pi \epsilon}}e^{-\frac{2}{3}((\epsilon+t)^3-t^3)}\mathbb{E}\left[\text{exp}\left(-2\sqrt{\epsilon^3} \int_{0}^1 \textbf{e}(z)dz \right)\right]+\limsup\limits_{x \uparrow 0} F^{\epsilon}(x)
\end{align*}
Now it suffices to see that 
\begin{align*}
\frac{2}{\sqrt{2\pi \epsilon}}e^{-\frac{2}{3}((\epsilon+t)^3-t^3)}\mathbb{E}\left[\text{exp}\left(-2\sqrt{\epsilon^3} \int_{0}^1 \textbf{e}(z)dz \right)\right]=
\frac{2}{\sqrt{2\pi \epsilon}}+O(\sqrt{\epsilon})\\=
\int_{\epsilon}^{\infty}\frac{1}{\sqrt{2\pi u^3}}du+O(\sqrt{\epsilon})
\end{align*}
By sending $\epsilon$ to zero we get the desired result. \end{proof}

We are now ready to prove the Theorem \ref{scalar}.
\begin{proof}[Proof of Theorem \ref{scalar}]
Recall that our solution is expressed as 
\begin{align*}
\rho(x,t)=L'\left(\frac{y(x,t)-x}{t}\right)=L'\left(\frac{\Psi^{tL(\frac{\cdot}{t})}(x)-x}{t}\right)
\end{align*}
Hence, $\rho$ is stationary by Theorem \ref{psi}, and so it is a time-homogenous Markov process, its generator is determined by 
\begin{align*}
\mathcal{A}^{t}\varphi(y)&=\lim_{h \rightarrow 0} \frac{\mathbb{E}[\varphi(\rho(h,t))-\varphi(\rho_{-}) \vert \rho(0,t)=\rho_{-}]}{h}\\
&=\lim_{h \rightarrow 0}\frac{\mathbb{E}[\varphi(L'(\frac{\Psi^{tL(\frac{.}{t})}(h)-h}{t}))-\varphi(\rho_{-}) \vert \Psi^{tL(\frac{\cdot}{t})}(0)=tH'(\rho_{-})]}{h}\\
&=-\frac{1}{t}L''(H'(\rho_{-}))\varphi'(\rho_{-})+\mathcal{A}^{tL(\frac{\cdot}{t})} \varphi(L'(\frac{\cdot}{t}))(tH'(\rho_{-}))\\
&=-\frac{\varphi'(\rho_{-}))}{tH''(\rho_{-})}+\mathcal{A}^{tL(\frac{.}{t})} \varphi(L'(\frac{\cdot}{t}))(tH'(\rho_{-}))\\
&=-\frac{\varphi'(\rho_{-})}{tH''(\rho_{-})}+\int_{\rho_{-}}^{\infty}(\varphi(\rho_{+})-\varphi(\rho_{-})) n(\rho_{-},\rho_{+},t)d\rho_{+}
\end{align*}
where 
\begin{align*}
n(\rho_{-},\rho_{+},t)=tH''(\rho_{+})\frac{j^{tL(\frac{\cdot}{t})}(tH'(\rho_{+}))}{j^{tL(\frac{\cdot}{t})}(tH'(\rho_{-}))}K^{tL(\frac{\cdot}{t})}(tH'(\rho_{-}),tH'(\rho_{+}))
\end{align*}
By a change of variables we have 
\begin{align*}
K^{tL(\frac{\cdot}{t})}(tH'(\rho_{-}),tH'(\rho_{+}))=\frac{\rho_{+}-\rho_{-}}{\sqrt{2\pi t^3(H'(\rho_{+})-H'(\rho_{-}))^3}} \times \\
\text{exp} \left (-\frac{t}{2}  \int_{\rho_{-}}^{\rho_{+}} (\rho_{*})^2 H''(\rho_{*})d\rho_{*} \right)
\mathbb{E}\left[\text{exp} \left(-\int_{\rho_{-}}^{\rho_{+}} \textbf{e}(tH'(\rho_{*}))d\rho_{*} \right)\right]
\end{align*}
Similarly
\begin{align*}
-j^{tL(\frac{\cdot}{t})}(tH'(\rho_{-}))=\rho_{-}+\int_{\rho_{-}}^{\infty}\frac{1-p(\rho_{-},\rho,t)}{\sqrt{2\pi t(H'(\rho)-H'(\rho_{-}))^3}} H''(\rho) d\rho
\end{align*}
where 
\begin{align*}
p(\rho_{-},\rho,t)=\text{exp} \left (-\frac{t}{2}  \int_{\rho_{-}}^{\rho} (\rho_{*})^2 H''(\rho_{*})d\rho_{*} \right)
\mathbb{E}\left[\text{exp} \left(-\int_{\rho_{-}}^{\rho} \textbf{e}(tH'(\rho_{*}))d\rho_{*} \right)\right]
\end{align*}
The theorem then follows by appropriately defining the kernel $K$. 
\end{proof}
\begin{remark}
While our main study focused on the case where the initial potential is a two-sided Brownian motion. It is not hard to see that we can extend the result about the \textit{profile} of the entropy solution when the potential is a spectrally positive L\'evy process with non-zero Brownian exponent. The main ingredients that were used were respectively the path decomposition of Markov processes at their ultimate maximum and the regularity properties of the transition function $f$. Both these facts hold true in the L\'evy case when the initial potential $U_0$ has a non-zero Brownian exponent, as the only difference in the Kolmogorov forward equation is an added integral operator accounting for the jumps of the L\'evy process. A similar approach will lead to the same smoothness property away from the singularity line $\{t=s\}$ (the presence of the heat operator $\partial_t -\frac{1}{2} \partial^2_{y}$ is key to have parabolic smoothing), which will allow all the operations in the second section to be valid. Moreover, one should be able to extract similar expression for the jump kernel $n$ by using the Girsanov theorem version for L\'evy processes. We chose in this paper to only discuss the Brownian motion case because it gives a general idea on how things work and also because it simplifies greatly the computations. One would expect to have similar formulas where the equivalent of the Brownian excursion will be the L\'evy bridge informally defined as a L\'evy process conditioned to stay positive and to start and end at zero. Those bridges are discussed in \cite{bravo}.
\end{remark}
\section{Structure of shocks of the entropy solution}\label{sect5}
A priori, from the involved expression of the generator in Theorem \ref{scalar}, one cannot easily claim whether if the structure of shocks of the solution $\rho$ is discrete or not. Indeed, this amounts to checking if the following integrability condition on the jump kernel $n$ holds
\begin{align*}
\lambda(\rho_{-})=\int_{\rho_{-}}^{\infty} n(\rho_{-},\rho_{+},t)d\rho_{+} < \infty \text{ for all } \rho_{-} \in \mathbb{R}
\end{align*}
However, using the recent theory of Lipschitz minorants of L\'evy processes developed in \cite{evansabramson} and \cite{evansouaki}, and following some of the arguments from the study of the structure of shocks in Burgers equation of \cite{abramson}, it turns out that when the initial potential is an \textit{abrupt} spectrally positive L\'evy process, one can prove that the set of jump times of the solution $\rho$ is discrete. 

As we did with Theorem \ref{scalar}, we will prove a general statement for the process $\Psi^{\phi}$ from which Theorem \ref{discrete} will follow. We state thus the following theorem 
\begin{thm}\label{discrete-psi}
Assume that $U_0$ is an \textit{abrupt} spectrally positive L\'evy process and $\phi$ is a strictly convex function such that $\lim_{\vert y \vert \to \infty} \vert \phi'(y) \vert =+\infty$ and $\displaystyle \lim_{\vert y \vert \rightarrow +\infty} \frac{U_0(y)}{\phi(y)}=0$ almost surely, then the range of $\Psi^{\phi}$ is a.s discrete. 
\end{thm}

\begin{proof}
From Theorem \ref{psi}, we know that for every $n \in \mathbb{Z}$ 
\begin{align*}
(\Psi^{\phi}(x+n)-n)_{x \in \mathbb{R}} \overset{\mathrm{d}}{=} (\Psi^{\phi}(x))_{x \in \mathbb{R}}
\end{align*}
hence it suffices to prove that the set $\text{range}(\Psi^{\phi}) \cap [0,1]$ is a.s discrete. Moreover, we can restrict the process $\Psi^{\phi}$ on $[-M,M]$. Indeed, we claim that the probability of the event 
\begin{align*}
A_M:=\{ \text{there exists } a \text{ such that } |a|\ge M \text{ and } \Psi^{\phi}(a) \in [0,1] \}
\end{align*}
goes to zero as $M$ goes to infinity. To show this claim, assume that there exists a sequence $(a_n)_{n \in \mathbb{N}}$ such that $\lambda_n:=\Psi^{\phi}(a_n) \in [0,1]$ and $\vert a_n \vert \to \infty$. By definition we have that 
\begin{equation}\label{ineq}
U_0(\lambda_n)-\phi(\lambda_n-a_n) \ge U_0(y)-\phi(y-a_n) \text{ for all } y
\end{equation}

Up to taking subsequences, we have either that $a_n \to \infty$ or $a_n \to -\infty$. If $a_n \to \infty$, take $y=a_n-1$ in \eqref{ineq}, then
\begin{equation}\label{inequ-1}
U_0(\lambda_n)-\phi(\lambda_n-a_n) \ge U_0(a_n-1)-\phi(-1)
\end{equation}
As $\phi'$ is strictly increasing, we must have $\lim_{y \rightarrow -\infty} \phi'(y)=-\infty$, and thus $\phi$ is decreasing for $y \rightarrow -\infty$. Hence from \eqref{inequ-1} and the fact that $\lambda_n \le 1$, we get
\begin{equation}\label{inequ-2}
U_0(\lambda_n)-U_0(a_n-1) \ge \phi(\lambda_n-a_n)-\phi(-1) \ge \phi(1-a_n)-\phi(-1)
\end{equation}
for $n$ large enough. However, because $(U_0(y))_{y\in \mathbb{R}}$ has the same distribution as $(-U_0((-y)-))_{y \in \mathbb{R}}$, then almost surely $\lim_{n \rightarrow \infty} \frac{U_0(a_n-1)}{\phi(1-a_n)}=0$, which is a contradiction with \eqref{inequ-2}. The case $a_n \rightarrow -\infty$ is similar by taking $y=a_n$ in \eqref{ineq}, proving thus our claim.\\

Define now the event $B_M$ as 
\begin{align*}
B_M=\left\{ \text{Card}\left(\text{range}(\Psi^{\phi}_{|[-M,M]|})\cap[0,1]\right)=\infty\right\}
\end{align*}
It suffices to prove that $\lim\limits_{M \rightarrow \infty} \mathbb{P}\left[B_M\right] =0$. \\

Suppose initially that $\mathbb{E}[|U_0(1)|]<\infty$ and let $C_M:=\sup\limits_{t \in [-2M,2M]} \vert \phi'(t) \vert $. Because of our assumption on $\phi$, then for $M$ large enough we have that $\mathbb{E}[\vert U_0(1) \vert]<C_M$. For any $a \in [-M,M]$ such that $\lambda_a:=\Psi^{\phi}(a) \in [0,1]$, we have for all $t \in [-M,M]$
\begin{equation}\label{inequ1}
U_0(t)-U_0(\lambda_a) \le \phi(t-a)-\phi(\lambda_a-a) \le C_M |t-\lambda_a|
\end{equation}
For $\alpha>0$ such that $\mathbb{E}[\vert U_0(1) \vert]<\alpha$, let us consider now the process $L^{\alpha}_0$ that is the $\alpha$-Lipschitz majorant of $U_0$, defined formally as 
\begin{align*}
L^{\alpha}_0(y)=\sup_{z \in \mathbb{R}} \left\{ U_0(z)-\alpha |z-y| \right \}
\end{align*}
We refer the reader to the two papers \cite{evansabramson} and \cite{evansouaki} for a detailed study of the Lipschitz minorant of a L\'evy process. Consider $G^{\alpha}_t$ (resp. $D^{\alpha}_t$) to be the last contact point before $t$ (resp. the first contact point after $t$) of $L^{\alpha}_0$ with $U_0$, i.e. 
\begin{align*}
G^{\alpha}_t = \sup \left \{ y<t : L_0^{\alpha}(y)=U_0(y) \right \} \text{ and } D^{\alpha}_t = \inf \left \{ y > t : L_0^{\alpha}(y)=U_0(y) \right \}
\end{align*} 
for any $t \in \mathbb{R}$. Moreover, let $\mathcal{Z}_{\alpha}$ be the contact set of $L^{\alpha}_0$ and $U_0$ defined as 
 \begin{align*}
 \mathcal{Z}_{\alpha}:=\{ y \in \mathbb{R} : L^{\alpha}_0(y)=U_0(y) \}
 \end{align*}
 Then on the event $\{ G^{C_M}_0, D^{C_M}_1 \in [-M,M]\}$, from the inequality \eqref{inequ1}, we have 
\begin{align*}
 U_0(G^{C_M}_0)-U_0(\lambda_a) \le C_M(\lambda_a -G^{C_M}_0) \text{ and } U_0(D^{C_M}_1)-U_0(\lambda_a) \le C_M(D^{C_M}_1-\lambda_a)
\end{align*}
Hence for $t \ge M$, we have 
\begin{align*}
U_0(t)-U_0(\lambda_a) &\le U_0(D^{C_M}_1)+C_M(t-D^{C_M}_1)-U_0(\lambda_a) \\
&\le C_M(D^{C_M}_1-\lambda_a)+C_M(t-D^{C_M}_1)=C_M\vert t-\lambda_a\vert
\end{align*}
Similarly for $t \le -M$ we get the same result. Together with \eqref{inequ1}, we deduce that for any $a \in [-M,M]$ such that $\lambda_a:=\Psi^{\phi}(a)\in [0,1]$, $\lambda_a$ is in the contact set $\mathcal{Z}_{C_M}$. However when $U_0$ is abrupt, we know from \cite{evansabramson}[See proof of Proposition 6.1] that this set is discrete, and hence $\mathcal{Z}_{C_M}\cap[0,1]$ is finite. Thus 
\begin{equation}\label{inequ dg}
\mathbb{P}[B_M] \le \mathbb{P}[G^{C_M}_0 \le -M] + \mathbb{P}[D^{C_M}_1 \ge M]
\end{equation} 
Now it is not hard to see that for $\alpha<\alpha'$, we have that $\mathcal{Z}_{\alpha} \subset \mathcal{Z}_{\alpha'}$. Hence, for $M$ large enough we have 
\begin{equation}\label{comparison}
D_1^{C_M} \le D_1^{\beta} , ~~ G^{C_M}_{0}\ge G^{\beta}_0
\end{equation}
where $\beta=\mathbb{E}[\vert U_0(1) \vert]+1$ is independent of $M$. However, from \cite{evansabramson}[Theorem 2.6] we know that the set $\mathcal{Z}_{\beta}$ is stationary and regenerative (see \cite{taksar} for the precise definition of stationary regenerative sets), thus the random variables $D^{\beta}_1-1$ and $-G^{\beta}_0$ have the same distribution as $D^{\beta}_0$. Moreover from \cite{evansabramson}[Equation (4.7)], we have that 
\begin{align*}
\mathbb{P}[D^{\beta}_0-G^{\beta}_0 \in dx]=\frac{x\Lambda^{\beta}(dx)}{\int_{\mathbb{R}+} x\Lambda^{\beta}(dx)}
\end{align*}
where $\Lambda^{\beta}$ is the L\'evy measure of the subordinator associated with the contact set $\mathcal{Z}_{\beta}$ (the stationarity of $\mathcal{Z}_{\beta}$ ensuring that $\int_{\mathbb{R}+} x\Lambda^{\beta}(dx)<\infty$). It follows thus from \eqref{comparison} that the right-hand side of \eqref{inequ dg} goes to zero when $M \rightarrow \infty$, from which we get the desired result that the range of $\Psi^{\phi}$ is discrete when $\mathbb{E}[\vert U_0(1) \vert]<\infty$.\\

Now, if $\mathbb{E}[\vert U_0(1) \vert]=\infty$, consider for any $N \in \mathbb{N}$  the truncated process $U_0^{N}$, that is the process $U_0$ started at zero and with its jumps of size greater than $N$ removed. It is formally defined as :

\begin{equation}\label{truncation}
U_0^{N}(y)=\left\{
    \begin{array}{ll}
       U_0(y)-\sum_{0 \le z \le y} (U_0(z)-U_0(z-))\ind_{\{ U_0(z)-U_0(z-) \ge N \}} \text{ if }~~ y \ge 0\\
       U_0(y)+\sum_{y \le z \le 0} (U_0(z)-U_0 (z-)) \ind_{\{ U_0(z)-U_0(z-) \ge N\}} \text{ if }  ~~ y \le 0
    \end{array}
\right.
\end{equation}

 We have that $\mathbb{E}[\vert U_0^{N}(1)\vert]<\infty$ as any L\'evy process with uniformly bounded jumps has finite moments of any order (see \cite{schilling}[Lemma 8.2]). Hence, if we denote by $\Psi^{\phi}_N$ the process $\Psi^{\phi}$ where we replace $U_0$ by $U_0^{N}$. By what we proved previously, we have that almost surely, the set $\text{range}(\Psi^{\phi}_N) \cap [0,1]$ is finite for every $N \in \mathbb{N}$ (as the finiteness of the moment of order $1$ of $U_0^{N}(1)$ ensures by the law of large numbers that $U_0^{N}(y)=o(\phi(y))$). By the arguments provided before, it suffices to prove that $\text{range}(\Psi^{\phi}_{|[-M,M]})\cap[0,1]$ is finite for every $M \ge 0$. 
Now, for $N \ge 1$ and $y \ge 0$ we have that 
\begin{align*}
|U_0^{N}(y)| &\le |U_0(y)|+\sum_{0 \le z \le y} (U_0(z)-U_0(z-))\ind_{\{ U_0(z)-U_0(z-) \ge N \}}\\
&\le |U_0 (y)| + |U_0(y)-U_0^{1}(y)| \le 2 |U_0(y)|+ |U_0^{1}(y)|
\end{align*}
and similarly for $y \le 0$. Thus almost surely
\begin{align*}
\lim_{\vert y \vert \to \infty} \sup_{N \ge 1} \frac{|U_0^{N}(y)|}{\phi(y \pm M)}=0
\end{align*}
as by the law of large numbers $U_0^{1}(y)=O(|y|)$. Let $K_1>0$ such that for all $|y| \le K_1$, we have almost surely
\begin{align*}
\sup_{N \ge 1} \frac{|U_0^{N}(y)|}{\phi(y \pm M)} \le \frac{1}{2}
\end{align*}
Let $K_2>0$ such that $y \mapsto \phi(y)$ is increasing on $[K_2,+\infty)$ and decreasing on $(-\infty,-K_2)$, then for $|y| \ge \max(K_1,K_2)+M$ and $a\in [-M,M]$, we have 
\begin{align*}
U_0^{N}(y)-\phi(y-a) &\le U_0^{N}(y)-\phi(y\pm M)\\
&\le -\frac{1}{2}\phi(y\pm M) \underset{y \to +\infty}{\rightarrow} -\infty
\end{align*}
Hence there exists $K>1$ large enough such that 
\begin{equation}\label{sup-sup}
\sup_{|y| \ge K} \sup_{a \in [-M,M]} \sup_{N \ge 1} \left(U_0^{N}(y)-\phi(y-a)\right) \le B:=\inf_{\lambda \in [0,1], a \in [-M,M]} \left(U_0(\lambda)-\phi(\lambda-a)\right)
\end{equation}
Now, the largest jump size of the process $U_0$ on any compact interval $[-R,R]$ is almost surely finite, because 
\begin{align*}
\mathbb{P}[\exists y \in [-R,R], U_0(y)-U_0(y-) \ge N] =1-e^{-2R \Pi([N,+\infty))} \underset{N \to \infty}{\rightarrow} 0
\end{align*}
where $\Pi$ is the L\'evy measure of $U_0$. Hence there exists a random $\tilde{N}$ such that $U_0^{\tilde{N}}(y)=U_0(y)$ on $[-K,K]$, and thus if $\text{range}(\Psi^{\phi}_{|[-M,M]})\cap[0,1]$ is infinite, then there exists infinitely many $\lambda_a \in [0,1]$ such that
\begin{align*}
U_0(\lambda_a)-\phi(\lambda_a -a) \ge U_0(y)-\phi(y-a) \text{ for all } y 
\end{align*}
which in light of \eqref{sup-sup} implies that 
\begin{align*}
U_0^{\tilde{N}}(\lambda_a)-\phi(\lambda_a -a) \ge U_0^{\tilde{N}}(y)-\phi(y-a) \text{ for all } y 
\end{align*}
and this is a contradiction with the fact that $\text{range}(\Psi^{\phi}_{\tilde{N}}) \cap [0,1]$ is finite, thus completing the proof.
\end{proof}
Finally, we are left to prove Theorem \ref{discrete}
\begin{proof}[Proof of Theorem \ref{discrete}]
In light of Theorem \ref{discrete-psi}, it suffices to check that for any $t>0$ we have 
\begin{align*}
\lim_{ \vert x \vert \to \infty} \left\vert L'\left(\frac{x}{t}\right) \right \vert = +\infty 
\end{align*}
However, due to the convexity of $L$, the function $L'$ is increasing and thus the limits,
\[
l^{+}:=\lim_{x \rightarrow \infty} L'(x) \text{ and } l^{-}:=\lim_{x \rightarrow -\infty} L'(x)
\]
exist. However, due to the superlinear growth of $H$ (and thus of $L$), it must be that $l^{+}=\infty$ and $l^{-}=-\infty$, which gives the desired result.
\end{proof}
\begin{remark}
The class of abrupt L\'evy processes mentioned in Theorem \ref{discrete} is quite large. Indeed, it contains any linear combination of Brownian motion with linear drift and stable L\'evy processes with index $\alpha \in (1,2)$ with its negative jumps removed.
\end{remark}

\bibliographystyle{abbrv}
\bibliography{scalarconvlaw}
\end{document}